\documentclass{opt2020} % Include author names
\usepackage{hyperref}
\usepackage{nicefrac}
\usepackage{upgreek,amsfonts,bm, bbm,enumerate,multirow,diagbox,makecell}
\usepackage[shortlabels]{enumitem} 

\usepackage{pseudocode,algpseudocode}   % Do NOT use algorithm.sty or algorithm2e.sty for floating (OK if use [H])
\usepackage{relsize} 
\usepackage{cases} % to use numcases  
\usepackage{lipsum} % to display a picture in two-column
\usepackage{caption} % to center caption
\usepackage{dblfloatfix} % to explicit location of table/graph 

\makeatletter 
\def\namedlabel#1#2{\begingroup
	#2%
	\def\@currentlabel{#2}%
	\phantomsection\label{#1}\endgroup
}
\makeatother

\newtheorem{thm}{Theorem}
\newtheorem{lem}{Lemma}

\newtheorem{assumeA}{Assumption A.\hspace*{-1.2mm}}

\newtheorem{rem}{Remark}

\newtheorem{manualtheoreminner}{Assumption A.\hspace*{-1.2mm}} 
\newenvironment{assumeAprime}[1]{%
	\renewcommand\themanualtheoreminner{#1}%
	\manualtheoreminner
}{\endmanualtheoreminner}

\newtheorem{manualtheoreprime}{Theorem} 
\newenvironment{thmPrime}[1]{%
	\renewcommand\themanualtheoreprime{#1}%
	\manualtheoreprime
}{\endmanualtheoreprime}

\newcommand{\remove}[1]{}
\newcommand{\given}[2]{\left. #1 \right| #2}
\newcommand{\ilparenthesis}[1]{( #1 )}
\newcommand{\ilbracket}[1]{[ #1 ]} 
\newcommand{\ilset}[1]{\{ #1 \}}
\def \set#1{\left \{#1 \right \}}
\newcommand{\bracket}[1]{\left[  #1  \right]}

\def \parenthesis#1{\left (#1 \right)}
\newcommand{\norm}[1]{\| #1 \|} 
\newcommand{\largenorm}[1]{\bigg\| #1 \bigg\|} 
\newcommand{\abs}[1]{\left| #1 \right|}

\newcommand{\indicator}{\mathbb{I}}

\newcommand{\diff}{\mathrm{d}} 

\newcommand{\real}{\mathbb{R}}

\newcommand{\field}{\mathcal{F}}

\usepackage{bbm}
\newcommand{\Prob}{\mathbbm{P}}
\newcommand{\E}{\mathbbm{E}}
\newcommand{\Var}{\mathrm{Var}}
\usepackage{relsize}
\newcommand{\transpose}{\mathsmaller{T}}
\newcommand{\tr}{\mathrm{tr}}

\usepackage{bm}
 
\newcommand{\zero}{\boldsymbol{0}}

\newcommand{\diag}{\mathrm{diag}}

\newcommand{\bA}{\bm{A}}
\newcommand{\bB}{\bm{B}}

\newcommand{\bH}{\bm{H}}
\newcommand{\bI}{\bm{I}}

\newcommand{\bM}{\bm{M}}

\newcommand{\bP}{\bm{P}}

\newcommand{\bg}{\bm{g}}

\newcommand{\bv}{\bm{v}}

\newcommand{\bx}{\bm{x}}

\usepackage{upgreek}
\renewcommand{\alpha}{\upalpha}
\renewcommand{\beta}{\upbeta}
\renewcommand{\gamma}{\upgamma}
\renewcommand{\delta}{\updelta}
\renewcommand{\epsilon}{\upepsilon}
\renewcommand{\varepsilon}{\upvarepsilon}
\renewcommand{\zeta}{\upzeta}
\renewcommand{\eta}{\upeta} 
\renewcommand{\theta}{\uptheta}
\renewcommand{\vartheta}{\upvartheta}
\renewcommand{\iota}{\upiota}
\renewcommand{\kappa}{\upkappa}
\renewcommand{\lambda}{\uplambda}
\renewcommand{\mu}{\upmu}
\renewcommand{\nu}{\upnu}
\renewcommand{\xi}{\upxi}
\renewcommand{\pi}{\uppi}
\renewcommand{\rho}{\uprho}
\renewcommand{\varrho}{\upvarrho}
\renewcommand{\sigma}{\upsigma}
\renewcommand{\tau}{\uptau}
\renewcommand{\upsilon}{\upupsilon}
\renewcommand{\phi}{\upphi}
\renewcommand{\varphi}{\upvarphi}
\renewcommand{\chi}{\upchi}
\renewcommand{\psi}{\uppsi}
\renewcommand{\omega}{\upomega} 

\newcommand{\bbeta}{\boldsymbol{\upbeta}}

\newcommand{\bDelta}{\boldsymbol{\Delta}} 

\newcommand{\bzeta}{\boldsymbol{\upzeta}}
\newcommand{\btheta}{\boldsymbol{\uptheta}}
\newcommand{\bvartheta}{\boldsymbol{\uptheta}^* }

\newcommand{\bLambda}{\boldsymbol{\Lambda}}

\newcommand{\bGamma}{\boldsymbol{\Gamma}}
\newcommand{\bmu}{\boldsymbol{\upmu}} 

\newcommand{\bxi}{\boldsymbol{\upxi}}

\newcommand{\bSigma}{\boldsymbol{\Sigma}}
\usepackage{mathrsfs}
\newcommand{\gain}{a}
\newcommand{\perturb}{c} 
\newcommand{\tperturb}{\widetilde{c}}
\newcommand{\loss}{L}   
\newcommand{\hbg}{\hat{\bg}}
\newcommand{\hbtheta}{\hat{\btheta}}

\newcommand{\bias}{\bbeta}

\newcommand{\algoName}[1]{\textsc{#1}} 
\newcommand{\obH}{\overline{{\bH}}}
\newcommand{\hbH}{\hat{\bH}}
\newcommand{\tbDelta}{\widetilde{\bDelta}}
\newcommand{\mapping}{\bm{m}} 
\newcommand{\noise}{\bxi} 

\newcommand{\Dimension}{d} 
\newcommand{\queryNum}{q}
\newcommand{\weight}{w}

\newcommand{\BoundThirdOrder}{D_1}
\newcommand{\BoundPerturbation}{D_2}
\newcommand{\BoundNoise}{D_3}
\newcommand{\BoundSecondOrder}{D_4}
\newcommand{\BoundFirstNoisy}{D_5}
\newcommand{\constNum}{C}
\newcommand{\noisyG}{\mathsf{\bg}}
\newcommand{\noisyGcomponent}{\mathsf{g}}
\newcommand{\noisyGcov}{\mathsf{\Sigma}}
\newcommand{\mappingPD}{\bm{f}}
\newcommand{\accuracy}{\upepsilon}
\newcommand{\individualLoss}{\mathrm{loss}}
\newcommand{\penaltyPara}{\upkappa}
\newcommand{\hyperPara}{\mathcal{S}}

\usepackage{color} 
\usepackage{xcolor}

\title[Hessian Inverse Approximation as Covariance for  Random Perturbation]{Hessian Inverse Approximation as Covariance for  Random Perturbation in Black-Box Problems}

\optauthor{\Name{Jingyi Zhu} \\ \Email{jingyi.zhu@alibaba-inc.com} \\
	\addr DAMO Academy, Alibaba Group, USA}

\begin{document}
	
	\remove{ 
	\textcopyright~12th International (Virtual) Workshop on ``Optimization for Machine Learning'' as a part of the NeurIPS 2020 conference.  Personal use of this material is permitted.  Permission from OPT2020 must be obtained for all other uses, in any current or future media, including reprinting/republishing this material for advertising or promotional purposes, creating new collective works, for resale or redistribution to servers or lists, or reuse of any copyrighted component of this work in other works.
	\thispagestyle{empty}
	\newpage
	
	\setcounter{page}{1}
}

\maketitle

\begin{abstract}  
	
		 In  stochastic optimization problems using  noisy zeroth-order (ZO) oracles only, 
	the      randomized counterpart   of Kiefer-Wolfowitz-type      method is widely used to estimate the gradient. 
	Existing algorithms 
	generate the randomized perturbation from a zero-mean and unit-covariance distribution. In contrast, this work considers the generalization where the perturbations have  a potentially  non-identity covariance  constructed from the history of the ZO queries. We propose  to feed  the second-order approximation into the covariance matrix of the random perturbation,  so it  is dubbed as Hessian-aided random perturbation (HARP).  
	HARP collects  four zeroth-order queries per iteration to  form   approximations  for both  the gradient  and the Hessian.
	We show the convergence (in an  almost surely sense) and derive the convergence rate for \algoName{HARP} under standard  assumptions. We demonstrate, with theoretical guarantees and numerical experiments, that \algoName{HARP} is less sensitive to ill-conditioning and more query-efficient than   other gradient approximation schemes using unit-covariance random perturbation.

\end{abstract}

\begin{keywords} 
       
       stochastic optimization, simultaneous perturbation, gradient-free methods, Hessian approximation
\end{keywords}

\section{Introduction} 
\label{sec:introduction}

 Stochastic approximation (SA) is a class of recursive procedures to locate roots of equations in the presence of  noisy measurements.  
When only noisy zeroth-order (ZO) information
is available, it is common  practice to generate deterministic perturbation \cite{kiefer1952stochastic,blum1954multidimensional}  or random perturbation  \cite{ermol1969method,katkovnik1972convergence,spall1992multivariate} in   
finding extrema. SA methods using ZO information have regained their popularity  in evolutionary strategy (as an alternative to reinforcement learning) \cite{salimans2017evolution,mania2018simple}  and adversarial image attack \cite{kurakin2016adversarial,carlini2017towards}.   To the best of our knowledge, all the existing random-perturbation-based methods generate the   perturbation from a distribution with zero-mean and unit-covariance, which enforce    that every component of the perturbation vector  is independent with all other components.  The resulting gradient estimate may not be robust to scaling and correlation of different parameters. Therefore,  this paper establishes the theoretical guarantee for the SA procedure using random perturbation with    non-identity covariance.  Specifically, we feed the Hessian inverse approximation into  the perturbation covariance, so the newly-proposed method is dubbed as Hessian-aided random perturbation (HARP).  HARP exhibits faster and more stable convergence performance   other SA algorithms   in  ill-conditioned problems, for which we provide both the theoretical analysis and the numerical illustration (via universal image attack).

 We now    describe the problem setting.   
  Let  $\btheta\in\real^{\Dimension}$  concatenate  all the adjustable   model parameters. 
Let the random variable $\upomega\in\Omega$ represent   the (generally uncontrollable) stochasticity of the underlying system.
Consider  
\begin{equation}\label{eq:SOsetup}
\min_{\btheta\in\real^\Dimension}\loss\ilparenthesis{\btheta} \equiv \E _{\upomega\sim\Prob} \ilbracket{\ell\ilparenthesis{\btheta,\upomega}}\,,
\end{equation} 
where the   loss function $\loss\ilparenthesis{\cdot}: \real^{\Dimension}\mapsto\real $   measures the  underlying system  performance,  and the random variable  $\ell\ilparenthesis{\cdot,\cdot}: \real^{\Dimension}\times\Omega \mapsto\real$ evaluated at   $ \ilparenthesis{\btheta,\upomega} $ 
represents  a  
noisy  observation of $ \loss\ilparenthesis{\btheta}$  corrupted by  $\upomega$. 
 Under the context that \emph{only} noisy zeroth-order (ZO) information $\ell\ilparenthesis{\cdot,\upomega}$ for some $\upomega\sim\Prob$ is available at certain values of $\btheta$ and that these noisy  ZO  queries  may be  \emph{expensive}, this work  considers 
 the    generic  stochastic approximation (SA) algorithm:
 \begin{equation}\label{eq:1stOrderSA}
 \hbtheta_{k+1} = \hbtheta_k - \gain_k \hbg_k\ilparenthesis{\hbtheta_k}\,,\,\, k\ge 1\,,
 \end{equation} where $\hbtheta_{k}$ denotes the  recursive estimate  at the \(k\)th iteration, $\hbg_k\ilparenthesis{\hbtheta_k}$ represents  the estimate for the gradient $\bg\ilparenthesis{\hbtheta_k}$, and $\gain_k>0 $ is the stepsize.     
 This work focuses  on  the following  gradient estimation scheme   using   two ZO queries per iteration: 
  \begin{align}\label{eq:ZOgradientTwoMeasurements}
 &\hbg_k\ilparenthesis{\hbtheta_k} = \frac{ \ell\ilparenthesis{\hbtheta_k+\perturb_k\bDelta_k,\upomega_{k}^{+}}  - \ell\ilparenthesis{\hbtheta_k-\perturb _k\bDelta_k, \upomega_k^{-}} }{2 \perturb _k}  \mapping_k\ilparenthesis{\bDelta_k}\,,
 \end{align} 
 where $\perturb_k$ represents  the   differencing  magnitude,    the $\Dimension$-dimensional random perturbation vectors $ \bDelta_k $ 
 is assumed to be drawn from a distribution with $\zero$-mean and $\bSigma_k^{-1}$-covariance, and the mapping  $\mapping_k\ilparenthesis{\cdot}:\real^{\Dimension}\mapsto\real^{\Dimension}$ is odd. The details  will be  discussed momentarily.

 As for  the statistical structure between $ \upomega_k^{+} $ and $ \upomega_k^{-} $, 
 two classical scenarios are considered.  The first one  where $ \upomega_k^{+} $ and $ \upomega_k^{-} $ are independent and identically distributed  will be termed  as  IID. The antithesis of IID, where $ \upomega_k^{+} = \upomega_k^{-}  $, will be  referred to as ``common random number''  (CRN). The CRN scenario are useful  \emph{simulation-based} optimization.

 \subsection{Prior Work and Our Contribution}
 \label{subsect:PriorWork} 
 
 The generic  form for gradient estimate in   (\ref{eq:ZOgradientTwoMeasurements}) subsumes  
 random direction stochastic approximation (\algoName{RDSA})   \cite{ermol1969method,ermoliev1983stochastic} with $\bDelta_k$ being uniformly distributed on the unit spherical surface and $\mapping_k\ilparenthesis{\bDelta_k} = \Dimension\bDelta_k$,  
 smoothed functional stochastic approximation (\algoName{SFSA}) \cite{katkovnik1972convergence} with $\bDelta_k$ being standard multivariate normally distributed and $\mapping_k\ilparenthesis{\bDelta_k} = \bDelta_k$, simultaneous perturbation stochastic approximation (\algoName{SPSA})   \cite{spall1992multivariate} with each component of $\bDelta_k$ being Rademacher distributed and $\mapping_k\ilparenthesis{\bDelta_k}  = \bDelta_k$.  Although the randomized scheme (\ref{eq:ZOgradientTwoMeasurements}) exists  for a long time and demonstrates numerical  advantages over \algoName{FDSA}  \cite{kiefer1952stochastic}, theoretical foundation regarding the \emph{optimal} choices of $\bDelta_k$ is lacking and extra caution is required in its implementation.

  We propose a new algorithm called ``Hessian-aided random perturbation'' (HARP).
The choice of feeding Hessian approximation into $\bSigma_k$ is motivated by overcoming the shortcomings of $\bSigma_k=\bI$ in Section~\ref{subsect:Motivation},   analyzed theoretically through almost surely convergence and convergence rate  in   Section~\ref{sect:CRN}, and demonstrated through two numerical experiments in Section~\ref{sect:Numerical}.  
Previously, 
in both stochastic optimization  \cite{spall2000adaptive}  and deterministic optimization, the Hessian  is applied in parameter update \emph{only}.  \algoName{HARP}    adaptively  changes    the covariance $\bSigma_k^{-1}$  of the perturbation $\bDelta_k$ using Hessian approximation, so that one can conveniently handle the issues   pertaining to   the scaling and correlation of different parameters, see Section~\ref{subsect:Motivation}. Compared with prior algorithms using unit-covariance random perturbation, \algoName{HARP} exhibits faster and more stable convergence performance, especially in ill-conditioned problems.

What sets our work different from the prior work  \cite{nesterov2017random,ghadimi2013stochastic} is       discussed in Subsection~\ref{subsect:Nesterov}. In short, \cite{nesterov2017random} is applicable for additive\footnote{This is a special CRN case, which completely remove the dependency on the observation noise from  the entire gradient estimate $\hbg_k\ilparenthesis{\hbtheta_k}$ in (\ref{eq:ZOgradientTwoMeasurements}). Namely, additive CRN noise, the numerator in (\ref{eq:ZOgradientTwoMeasurements}) involves the randomness in $\bDelta_k$ only.   }  CRN noise, and the corresponding analysis cannot be generalized to the general CRN noise discussed in Section~\ref{sect:CRN}, not to mention the IID noise discussed in Section~\ref{sect:IID}. Additionally,  \cite{ye2018hessian} also considers leveraging the Hessian estimates to achieve faster convergence. The results therein have to be interpreted with extra caution: the random perturbation $\bDelta_k$ impacts  both the gradient and the Hessian estimates at each iteration, yet the proofs therein ignore the randomness in the Hessian estimate.

  \subsection{Notation Convention}\label{subsect:Notation}
   
     \paragraph{Matrix and vector operations} Let $\bA\in\real^{\Dimension\times\Dimension}$ be a matrix and let $ \bx\in\real^{\Dimension} $ be a vector.   
   $ \norm{\bx} $ returns the Euclidean norm of $\bx$, and $\norm{\bA}$ returns the spectral norm of $\bA$.     
   If  $\bA$ is real-symmetric, $ \uplambda_{\min}\ilparenthesis{\bA} $ and $ \uplambda_{\max}\ilparenthesis{\bA}  $ return the smallest and the largest eigenvalues of $\bA$.   The binary operator $\otimes$ represents the Kronecker product.

       \paragraph{Probability and SA conventions}
   For a sequence  of random variables $X_k$, $X_k = o_P(1)$ means that $X_k$ converges to zero in probability as $k\to\infty$, and $ X_k= O_P(1) $ means that $ X_k $ is stochastically bounded.

     Let $\field_k$ represent the history of the recursion (\ref{eq:1stOrderSA}) until the \(k\)th iteration, and let $ \E_k\ilparenthesis{\cdot} $ denote the conditional expectation $ \E\ilbracket{\given{\cdot}{\field_k}} $.   To facilitate later discussion, we rewrite  $\hbg_k\ilparenthesis{\hbtheta_k}$  as:
  \begin{align}
  	 \hbg_k\ilparenthesis{\hbtheta_k} 
  	&\,\, = \bg\ilparenthesis{\hbtheta_k} + \E_k\ilbracket{\hbg_k\ilparenthesis{\hbtheta_k} - \bg\ilparenthesis{\hbtheta_k}} + \ilset{ \hbg_k\ilparenthesis{\hbtheta_k} - \E_k\ilbracket{\hbg_k\ilparenthesis{\hbtheta_k}} }\nonumber\\
  	&\,\,\equiv \bg\ilparenthesis{\hbtheta_k} + \bias_k\ilparenthesis{\hbtheta_k} + \noise_k\ilparenthesis{\hbtheta_k}\,, \label{eq:gDecomposition}
  \end{align}
  where $\bias_k\ilparenthesis{\hbtheta_k}$ represents  the bias of $\hbg_k\ilparenthesis{\hbtheta_k}$ as an estimator of $\bg\ilparenthesis{\hbtheta_k}$, and $\noise_k\ilparenthesis{\hbtheta_k}$ represents the noise term.  The decomposition (\ref{eq:gDecomposition}) is useful in asymptotic normality of $\hbtheta_k$  later, where $\bias_k\ilparenthesis{\hbtheta_k}$ affects the asymptotic mean and $\noise_k\ilparenthesis{\hbtheta_k}$ affects the asymptotic variance.

   \paragraph{Miscellaneous notation }     $\indicator_E$ represents the indicator function of a logical expression $E$.  
  For  $\loss\ilparenthesis{\cdot}$ that is three-times continuously differentiable, let  $ \loss^{(3)}\ilparenthesis{\btheta} \in\real^{1\times \Dimension^3}  $
  represent the third-order derivative  of $\loss\ilparenthesis{\cdot}$ evaluated at $\btheta$; moreover, let $ L^{(3)}_{i_1,i_2,i_3}\ilparenthesis{\btheta} \in\real $ represent the third-order derivative of $\loss\ilparenthesis{\cdot}$  with respect to (w.r.t.) the \(i_1\)th, \(i_2\)th,  and \(i_3\)th elements  of $\btheta$.

\section{    Motivation and Description of   \algoName{HARP}} \label{sect:Description}
This section 
motivates \algoName{HARP} and lists the   pseudo code.

\subsection{Motivation Behind \algoName{HARP}} \label{subsect:Motivation}

Prior work  summarized  Section~\ref{subsect:PriorWork}    enforce  $\bSigma_k=\bI$. We take  \algoName{SPSA} as an example.

\begin{enumerate}[(i)]
	\item \label{item:scaling}
	The estimate $\hbg_k\ilparenthesis{\hbtheta_k}$ using \algoName{SPSA} may not be robust to scaling, as every component of $\hbtheta_k$ is perturbed by  the \emph{same} magnitude of perturbation $\perturb_k$.    
	\item \label{item:correlation}
	The estimate $\hbg_k\ilparenthesis{\hbtheta_k}$ may not be robust to various correlations between different components of the parameter, as the perturbations  along all components  of $\btheta$ are \emph{independent} with each other. 
	
\end{enumerate}

As it turns out later   in Subsection~\ref{subsect:HARP}, a sensible  choice of $\bSigma_k$ is $ {\bH\ilparenthesis{\hbtheta_k}}$.  In this way,
\ref{item:scaling} can be resolved:  say,  $\loss\ilparenthesis{\btheta} = (100\uptheta_1^2+\uptheta_2^2)/2$, then a $ \bDelta_k $ with zero mean and a covariance of  $ \diag\ilparenthesis{0.01, 1} $ will \emph{on average} impose $10\%$ of the change magnitude in $\uptheta_2$ onto that of  $\uptheta_1$. Meanwhile, \ref{item:correlation} can be handled: say,   $ \loss\ilparenthesis{\btheta} = \ilparenthesis{\uptheta_1^2 + \uptheta_2^2 + \uptheta_1\uptheta_2}/2 $. When the direction of the gradient estimate, $\mapping_k\ilparenthesis{\bDelta_k}$, has a covariance of $ \begin{pmatrix}
1 &0.5 \\ 0.5 & 1
\end{pmatrix} $ will have heavier probability mass on $ \ilparenthesis{1,\,1}^\transpose $ and $ \ilparenthesis{-1,\,-1}^\transpose $ than on $ \ilparenthesis{-1, \, 1}^\transpose $ and $ \ilparenthesis{1,\,-1}^\transpose $.

    \subsection{Algorithm Description}\label{subsect:decription}  
Section~\ref{subsect:Motivation}  explains the   shortcomings \ref{item:scaling}--\ref{item:correlation} of using $\bSigma_k = \bI$   and illustrates the benefit of using $\bSigma_k = \bH\ilparenthesis{\hbtheta_k}$. Obviously, we cannot access $\bH\ilparenthesis{\hbtheta_k}$ in the black-box problem (\ref{eq:SOsetup}). We  consider constructing   estimate  for $\bH\ilparenthesis{\hbtheta_k}$ using ZO queries gathered in $\field_k$.  To form a $\field_k$-measurable second-order approximation,  \algoName{HARP}  is comprised of two recursions, one for $\btheta$ as in (\ref{eq:1stOrderSA}) and one for the Hessian $\bH\ilparenthesis{\btheta}$ as below:
 \begin{equation}\label{eq:HessianUpdate} 
 \begin{cases}
 & \hbH_k = \mappingPD_k\parenthesis{\obH_k  }\,,  \\
 &\obH_ 
 {k+1} =     \ilparenthesis{1-\weight_k} \obH_k + \weight_k  \set{ \mapping_k\ilparenthesis{\tbDelta_k} \ilbracket{\mapping_k\ilparenthesis{\bDelta_k}}^\transpose + \mapping_k\ilparenthesis{\bDelta_k} \ilbracket{\mapping_k\ilparenthesis{\tbDelta_k}}^\transpose }\bar{\ell}_k/\ilparenthesis{4\perturb_k\tperturb_k}\,. 
 \end{cases}
\end{equation} 
Here,  $\perturb_k$ and $\tperturb_k$  are the   differencing  magnitudes,    the $\Dimension$-dimensional random perturbation vectors $ \bDelta_k $ and $\tbDelta_k$
are assumed to be drawn from a distribution with $\zero$-mean and $\bSigma_k^{-1}$-covariance,  the mapping  $\mapping_k\ilparenthesis{\cdot}:\real^{\Dimension}\mapsto\real^{\Dimension}$ is odd,  and  $ \overline{\ell}_k = \ell\ilparenthesis{\hbtheta_k+\perturb_k\bDelta_k + \tperturb_k\tbDelta_k, \upomega_k^{+,+}} - \ell\ilparenthesis{\hbtheta_k + \perturb_k\bDelta_k,\upomega_k^+} - \ell\ilparenthesis{\hbtheta_k-\perturb_k\bDelta_k + \tperturb_k\tbDelta_k, \upomega_k^{-,+}} + \ell\ilparenthesis{\hbtheta_k-\perturb_k\bDelta_k,\upomega_k^+} $. The current estimate $\hbH_k$ and the smoothing (moving average) estimate  $\obH_k$ can be initialized as   the identity/scalar matrix.   The mapping      $ \mappingPD_k : \real^{\Dimension\times\Dimension} \mapsto\ilset{\text{positive definite matrices in }\real^{\Dimension\times\Dimension}} $ copes the potential nonpositive-definiteness of $\obH_k$. A valid choice for $ \mappingPD_k\ilparenthesis{\cdot} $ is $ \mappingPD_k\ilparenthesis{\bH} = \ilparenthesis{\bH ^\transpose\bH + \upvarepsilon_k\bI }^{\nicefrac{1}{2}} $ with $\upvarepsilon_k\to 0$, which can be implemented in $O(\Dimension^2)$ FLOPs \cite{zhu2019efficient}.  When $\tperturb_k = O(\perturb_k)$, and other gain sequence conditions are met, $\hbH_k $ approaches the Hessian evaluated  at the optimum at a rate no slower than $O(\perturb_k^2)$.

The detailed pseudo code   for  \algoName{HARP} is summarized in Algorithm~\ref{algo:HARP}. Readers are referred to our code hyperlinks in Section~\ref{sect:Numerical} and \cite{zhu2019efficient} for       implementation guidance.

\begin{algorithm2e}[!htbp]
	\caption{  Hessian-Amended Random Perturbation  (\href{https://github.com/jingyi-zhu/Fast2SPSA/tree/master/HARP}{GitHub})}
	\label{algo:HARP}
	\SetAlgoLined
	\KwResult{terminal estimate $\hbtheta_K$}
	initialization $\hbtheta_0$, $\hbH_0=\bI$, $\obH_0=\bI$, and  coefficients $ \gain_k, \perturb_k,\tperturb_k,\weight_k $ for $ 0\le k\le K $ \;
\remove{	\While{While condition}{
		instructions\;
		\eIf{condition}{
			instructions1\;
			instructions2\;
		}{
			instructions3\;
		}
	}}
	\For{$k = 0,1, \cdots K$}{
		generate $\bDelta_k$ from a distribution with a mean of $\zero$ and a covariance of $ \hbH_k ^{-1} $ and compute $ \mapping_k\ilparenthesis{\bDelta_k} =\hbH_k \bDelta_k $\;
		collect two ZO queries and estimate $\hbg_k\ilparenthesis{\hbtheta_k}$ via (\ref{eq:ZOgradientTwoMeasurements})\;
		update $\hbtheta_k$ using (\ref{eq:1stOrderSA})\;
		generate $\tbDelta_k$ from a distribution with a mean of $\zero$ and a covariance of $ \hbH_k ^{-1} $ and compute $ \mapping_k\ilparenthesis{\tbDelta_k} = \hbH_k \tbDelta_k $\;
		collect two additional ZO queries and estimate $\hbH_k $. \Comment{\cite[Algorithms 1--2]{zhu2019efficient} provides a way to achieve  $O(\Dimension^2)$ FLOPs. Other forms of $\mappingPD_k\ilparenthesis{\cdot}$ satisfying conditions in \cite{spall2000adaptive} also work.}
	}
\end{algorithm2e}

 \section{    Performance Metric}\label{sect:Metric}
 Before     analyzing  \algoName{HARP} listed  in Algorithm~\ref{algo:HARP}, let us   discuss the metric that evaluates the algorithm performance.  
 
 \subsection{    Convergence Mode }\label{subsect:mode}
 Now that all randomness in $\hbtheta_k$ stemming from $\Omega\times\Omega_{\bDelta}$, it is standard practice to measure the algorithmic performance of the recursions  (\ref{eq:1stOrderSA}) by showing 
 \begin{equation}\label{eq:StrongConvergence}
 \text{$\hbtheta_k$ converges almost surely (strongly) to $\bvartheta$\,, }
 (\hbtheta_k \stackrel{\mathrm{a.s.}}{\longrightarrow} \bvartheta)\,, 
 \end{equation}
 or \begin{equation}
 \label{eq:MeanSquaredConvergece} \text{$\hbtheta_k$ converges to $\bvartheta$ in mean-squared sense\,, } (\hbtheta_k \stackrel{\mathrm{m.s.}}{\longrightarrow} \bvartheta)\,.
 \end{equation} \citet{robbins1951stochastic} gave conditions for (\ref{eq:StrongConvergence}) whereas \citet{blum1954multidimensional} for (\ref{eq:MeanSquaredConvergece})\footnote{Neither (\ref{eq:StrongConvergence}) nor (\ref{eq:MeanSquaredConvergece}) implies the other \cite[Chap. 5]{billingsley2013convergence}. Both  (\ref{eq:StrongConvergence})  and  (\ref{eq:MeanSquaredConvergece})  imply  convergence in probability  and convergence in distribution.}. 
 We will prove (\ref{eq:StrongConvergence}) in Section~\ref{sect:Convergence}.

 \subsection{Rate of Convergence}\label{subsect:RateOfConvergence}
 When either   (\ref{eq:StrongConvergence}) or  (\ref{eq:MeanSquaredConvergece}) is 
 shown, finding the   rate of convergence naturally follows.  The asymptotic root-mean-squared (RMS) error  $ \ilbracket{\E\ilparenthesis{\norm{\hbtheta_k-\bvartheta}}^2}^{\nicefrac{1}{2}} $ of the underlying estimate $\hbtheta_k$ is a sensible measure  of the   distance between the     $\hbtheta_k$ and     $\bvartheta$ average across all sample paths.  Therefore, we aim to find the smallest  upper bound $\uptau^*$ such that  $ k^{\nicefrac{\uptau_0}{2}}\ilparenthesis{\hbtheta_k-\bvartheta} = O_P(1)$ for all $ \uptau_0 \le \uptau^*$, which is formalized as:  
 \begin{equation}\label{eq:NormalityProblem0}
 \begin{cases}
 &
 \max_{\hyperPara} \,\, \uptau  \,,\\ 
 &  \text{s.t. }     \text{random   sequence   $ \ilparenthesis{\hbtheta_k-\bvartheta} $ is $ O_P(k^{-\nicefrac{\uptau}{2}}) $}\,,
 \end{cases}
 \end{equation} 
 where the hyperparameter set $\hyperPara $ includes all the controllable stepsizes, and   both $\uptau$ and $O_P(1)$ are functions of   $\hyperPara$.   Thanks to the algorithmic form (\ref{eq:1stOrderSA}),  the decomposition  (\ref{eq:gDecomposition}), and \cite[Sect. 27]{billingsley2008probability}, the   constraint  in (\ref{eq:NormalityProblem0}) always takes the following form: 
 \begin{align} & 
 k^{\nicefrac{\uptau}{2}} \ilparenthesis{\hbtheta_k-\bvartheta}  
 \stackrel{\mathrm{dist.}}{\longrightarrow} \mathcal{N}\ilparenthesis{\bmu, \bB}   \text{ for   finite }\bmu, \bB\succ\zero\,, \,\, \label{eq:NormalityProblem}
 \end{align} 
 where $ \stackrel{\mathrm{dist.}}{\longrightarrow} $ represents  ``convergence in distribution,'' and  $ \ilparenthesis{\uptau, \bmu,\bB} $   are functions of $\hyperPara$.  When  (\ref{eq:NormalityProblem}) holds and  $ \ilbracket{k^{\nicefrac{\uptau}{2}}  \ilparenthesis{\hbtheta_k-\bvartheta} }  $ is uniformly  integrable for any $\uptau \le  \uptau^*$,  (\ref{eq:MeanSquaredConvergece}) holds.    The      RMS error  is asymptotic to  $\lim_{k\to\infty}\ilbracket{ \E\ilparenthesis{\norm{\hbtheta_k-\bvartheta} ^2}}^{\nicefrac{1}{2}} =  k^{-\nicefrac{\uptau}{2}}   \ilbracket{\norm{\bmu}^2 + \tr\ilparenthesis{\bB }}$.

 \subsubsection{Further Remarks on RMS}\label{subsect:Convergence} To minimize the RMS, it makes more sense to perform 
 \begin{equation}\label{eq:RMSproblem}
 \min_{\hyperPara} \set{k^{-\nicefrac{\uptau}{2}} \bracket{\norm{\bmu }^2+\tr\ilparenthesis{\bB}}}\,,  
 \end{equation}  as opposed to (\ref{eq:NormalityProblem0}).  When $k$ is small, the finite constant $ \ilbracket{\norm{\bmu}^2+\tr\ilparenthesis{\bB}} $ that are hidden from the  big-$O$ notation $O(k^{-\nicefrac{\uptau}{2}})$ can be dominating. For sufficiently large $k$,  the effect of the scaling coefficients dies down, and     (\ref{eq:RMSproblem}) reduces to       (\ref{eq:NormalityProblem0}). Sections~\ref{sect:IID}--\ref{sect:CRN} show  that the solution to (\ref{eq:NormalityProblem0}) is
 \begin{numcases}{\uptau^* = }
 \nicefrac{2}{3}\,, &  \text{ for IID noise}\,, \label{eq:OptimalRateCMC}  \\
 1\,, & \text{ for CRN noise}\,, \label{eq:OptimalRateCRN}
 \end{numcases}
 when $\loss\ilparenthesis{\cdot}$ is
 is non-quadratic\footnote{For a  quadratic function $\loss\ilparenthesis{\cdot}$,  $\uptau^*=1$ for both IID and CRN noise.  } and    
 three-times\footnote{For a function $\loss\ilparenthesis{\cdot}$ that is \(p\)-times continuously differentiable for odd $p$,  the fastest rate     for the RMS  is $ O\ilparenthesis{k^{-\nicefrac{\ilparenthesis{p-1}}{2p}}} $, which goes to $ O(k^{-\nicefrac{1}{2}}) $ as $p\to\infty$ \cite{fabian1971stochastic}.
 } continuously differentiable.

 \subsubsection{Iteration and Query Complexity}
 The complexity analysis for (\ref{eq:1stOrderSA}) is straightforward when the RMS metric (\ref{eq:RMSproblem}) is in use. 
 To achieve  
 \begin{equation}\label{eq:accurateSolution}
 \accuracy\text{-accurate estimate }\hbtheta_k\text{ s.t. } \ilbracket{ \E\ilparenthesis{\norm{\hbtheta_k-\bvartheta} ^2}}^{\nicefrac{1}{2}} \le \accuracy\,, 
 \end{equation}
 the the \emph{average} desired number of iteration is 
 \begin{equation}
 \label{eq:iterationComplexity} \set{ \nicefrac{\ilbracket{\norm{\bmu}^2+\tr\ilparenthesis{\bB}}}{\accuracy} }^{\nicefrac{2}{\uptau^*}} = \begin{cases}
 O(\accuracy^{-3})\,, & \text{   IID noise\,,}\\
 O(\accuracy^{-2})\,, & \text{   CRN noise\,.}
 \end{cases}
 \end{equation}
 
 \begin{rem}\label{rem:MultipleQueries}
 	When  (\ref{eq:1stOrderSA})    takes a fixed number, say $2\queryNum$, of ZO queries, and $\queryNum$ is independent from the parameter dimension $\Dimension$,  
 	then the corresponding query complexity is  
 	\begin{equation}
 	\label{eq:queryComplexity} 2 \queryNum
 	\set{ \nicefrac{\ilbracket{\norm{\bmu}^2+\tr\ilparenthesis{\bB/\queryNum}}}{\accuracy} }^{\frac{2}{\uptau^*}} = \begin{cases}
 	O(\accuracy^{-3}), & \text{IID noise,}\\
 	O(\accuracy^{-2}), & \text{CRN noise.}
 	\end{cases}
 	\end{equation}
 \end{rem}

 \subsection{Other Forms of    ``Convergence''  Rate }\label{subsect:Nesterov}

 \cite[Sect. 4]{nesterov2017random} 
 uses the following notion   
 \begin{equation}\label{eq:accurateSolution-Nesterov}
 \accuracy\text{-accurate estimate }\hbtheta_k\text{ s.t. }  \E \ilbracket{\loss\ilparenthesis{\hbtheta_k} - \loss\ilparenthesis{\bvartheta}}\le \accuracy\,, 
 \end{equation}
 as opposed to (\ref{eq:accurateSolution}), and (\ref{eq:accurateSolution-Nesterov}) is popular for analyzing ZO algorithms    \cite{ghadimi2013stochastic}. 
 Let us offer a few remarks on the differences between (\ref{eq:accurateSolution})  and (\ref{eq:accurateSolution-Nesterov}).  First of all, the  resultant ``convergence'' rate under the notion (\ref{eq:accurateSolution-Nesterov}) require \emph{non}-decaying rate.   \cite[Chap. 4]{zhu2020error}
 points out that $\hbtheta_k$ will \emph{not} converge to $\bvartheta$ in standard statistical sense (either a.s. or m.s. in Subsection~\ref{subsect:mode}) when $\gain_k\not\to 0$. In fact, there is no ``convergence'' per se \cite{zhu2020stochastic}, as $\hbtheta_k$ will be ``random-walking'' within a  neighborhood of $\bvartheta$ even for sufficiently large $k$ \cite{zhu2018probabilistic}. 
 Second, \cite{nesterov2017random,ghadimi2013stochastic} and all the subsequent work on ZO algorithms require     \emph{additive} CRN noise, and the corresponding analysis   can \emph{not} be generalized to the general CRN noise case discussed in Section~\ref{sect:CRN}, not to mention  the IID noise case in Section~\ref{sect:IID}. 
 Third, the complexity result \cite[Eq. (59)]{nesterov2017random} does not reveal the eigen-structure of $\bH\ilparenthesis{\cdot}$ under certain smoothness assumption. On the contrary, $\bB$ in (\ref{eq:NormalityProblem}) conveys all the eigen-information of $\bH\ilparenthesis{\bvartheta}$, as we shall see momentarily. It   makes more  sense that the RMS should be larger for ill-conditioned problems compared with  well-conditioned problems.  
 Last but not least, $ \ilbracket{ \E\ilparenthesis{\norm{\hbtheta_k-\bvartheta} ^2}}^{\nicefrac{1}{2}} \le \accuracy  $ implies $\E \ilbracket{\loss\ilparenthesis{\hbtheta_k} - \loss\ilparenthesis{\bvartheta}} \le \accuracy'$, but generally not the other way around. 
 
 Overall, the notion (\ref{eq:accurateSolution-Nesterov}) and the analysis in 
 \cite{nesterov2017random,ghadimi2013stochastic}  are useful when (i) \emph{additive} CRN noise scenario is possible, and (ii) the experimenter   aims  to report  an acceptable output within the neighborhood of $\bvartheta$  given a limited iteration/query complexity. In fact, the non-decaying gain does provide better performance under a  budget-limited context \cite{zhu2020stochastic,zhu2016tracking}. Finally, it is advisable to use ``concentration'' and ``concentration rate''   \cite[Chaps. 7--8]{kushner2003stochastic}. 
 
 \subsection{  Dependency on Dimensionality $ \Dimension $}\label{subsect:Dimensionality}
 When dimensionality $\Dimension$  varies as the recursion goes on, e.g., $\Dimension$ plays an important role in structural optimization and etc., it is advisable to include the dependency on $\Dimension$ in the constraint of (\ref{eq:NormalityProblem0}) as ``random vector sequence   $ \ilparenthesis{\hbtheta_k-\bvartheta} $ is $ O_P(\Dimension k^{-\nicefrac{\uptau}{2}}) $.'' 
 Nevertheless,
 we decide to omit $\Dimension$  
 for clarity and for the reason that the problem dimension $\Dimension$ is \emph{generally} not an adjustable\footnote{
 	This contrasts with $k$ and $\accuracy$, both of which can be  selected  by the experimenter. }. Moreover, reducing the order of $\Dimension$  appearing in the convergence rate  is only possible when certain sparsity conditions are imposed  or certain sparsity-promoting regularization is added to the loss function.

 \section{Convergence Result}\label{sect:Convergence}
 
 \subsection{IID Scenario }
 \label{sect:IID} Overall,  
 the \emph{fastest} rate   of  the RMS error $  \ilparenthesis{ \E\ilbracket{ \norm{\hbtheta_k-\bvartheta}^2 } }^{\nicefrac{1}{2}} $ under IID noise   is 
 $O(k^{-\nicefrac{1}{3}})$,  
 which is attained when  $\gain_k =O(k^{-1})$ and $\perturb_k= O(k^{-\nicefrac{1}{6}})$.  (\ref{eq:OptimalRateCMC}) is inherently slower than (\ref{eq:OptimalRateCRN}), due to the trade-off  between   the bias magnitude $\E  \norm{\bias_k\ilparenthesis{\hbtheta_k}} $ and the variance $ \E\norm{\noise_k\ilparenthesis{\hbtheta_k}}^2 $ of the noise, which  is summarized in Lemma~\ref{lem:Tradeoff} below.

 As pointed out in  Subsection~\ref{subsect:Convergence}, not only the rate itself but also the scaling coefficient play a role  in  the algorithmic performance. This section  first show the a.s. convergence of the estimate $\hbtheta_k$ generated from (\ref{eq:1stOrderSA}) when the covariance of the perturbation sequence may be varied, and then    
 discuss the impact of the perturbation covariance on the finite constant $ \ilbracket{\norm{\bmu}^2+\tr\ilparenthesis{\bB}} $.

 \subsubsection{Order of Bias and Variance of  $\hbg_k\ilparenthesis{\hbtheta_k}$ }
 Let us first discuss the bias-variance trade-off in $\hbg_k\ilparenthesis{\hbtheta_k}$ for IID noise. Several assumptions are imposed on the underlying loss function $\loss\ilparenthesis{\cdot}$, the procedure to generate random perturbation $\bDelta_k$, especially the $\field_k$-measurable covariance  matrix  $\bSigma_k$, and the observation noise  $ \upvarepsilon_{k}^{\pm}\equiv \ell\ilparenthesis{\hbtheta_k\pm\perturb_k\bDelta_k,\upomega_k^{\pm}} - \loss\ilparenthesis{\hbtheta_k\pm\perturb\bDelta_k} $.

 \begin{assumeA}
 	[Loss Function] \label{assume:Loss} Assume  that  there exists some $K$, such that  for  $k\ge K$, 
 	$ \loss^{(3)} \ilparenthesis{\btheta} $
 	evaluated   for all $\btheta$ in an open neighborhood of $\hbtheta_k$ exists continuously and $ \norm{ \loss^{(3)} \ilparenthesis{\btheta}}_\infty \le  \BoundThirdOrder $   almost surely (a.s.). 
 \end{assumeA}
 \begin{assumeA}
 	[Perturbation] \label{assume:Perturbation} Assume that the perturbation sequence $ \ilset{\bDelta_k} $ are independently distributed  with a mean of $\zero$ and a covariance matrix $ \bSigma_k^{-1} $. Meanwhile, the mapping $ \mapping_k\ilparenthesis{\cdot} $   is an  odd function. Moreover,
 	both $\bDelta_k$ and $\mapping_k\ilparenthesis{\bDelta_k}$ are independent of $ \hbtheta_k $.  Finally, assume that $ \E_k \ilbracket{\mapping_k\ilparenthesis{\bDelta_k} \bDelta_k} \stackrel{\mathrm{a.s.}}{=} \bI $ and  $ \E_k\ilbracket{\norm{\bDelta_k}^6 \norm{\mapping_k\ilparenthesis{\bDelta_k}} ^2} \stackrel{\mathrm{a.s.}}{\le } \BoundPerturbation $ uniformly for all $k$.
 \end{assumeA}

 \begin{assumeA}
 	[IID] \label{assume:Noise}  Assume $ \E\ilbracket{\given{\upvarepsilon_k^{+} - \upvarepsilon_k^{-}}{\hbtheta_k,\bDelta_k}} \stackrel{\mathrm{a.s.}}{=} 0 $,  and $ \E\ilbracket{\given{\ilparenthesis{\upvarepsilon_k^{+} - \upvarepsilon_k^{-}}^2}{\hbtheta_k,\bDelta_k}} \stackrel{\mathrm{a.s.}}{\le }\BoundNoise $  uniformly for all $k$. 
 \end{assumeA}

 \begin{lem}
 	\label{lem:Tradeoff} When assumptions  A.\ref{assume:Loss}, A.\ref{assume:Perturbation}, and A.\ref{assume:Noise} hold, 
 	\begin{align}
 	&  	\bias_k\ilparenthesis{\hbtheta_k}\stackrel{\mathrm{a.s.}}{=}  \frac{\perturb_k^2}{12}\E _k \set{ \ilbracket{ \loss^{(3)} \ilparenthesis{\overline{\btheta}_k^+} + \loss^{(3)} \ilparenthesis{\overline{\btheta}_k^{-}} } \ilparenthesis{\bDelta_k\otimes\bDelta_k\otimes\bDelta_k} \mapping_k\ilparenthesis{\bDelta_k}  }, \label{eq:bias} \\
 	&  \noise_k\ilparenthesis{\hbtheta_{k}}   \stackrel{\mathrm{a.s.}}{=}  \frac{\ilparenthesis{\upvarepsilon_k^{+} - \upvarepsilon_k^{-}}}{2\perturb_k}\mapping_k\ilparenthesis{\bDelta_k}  +  \bracket{\mapping_k\ilparenthesis{\bDelta_k} \bDelta_k^\transpose-\bI} \bg\ilparenthesis{\hbtheta_k}  \nonumber
 	\\
 	& \,\, \quad\quad \quad \quad + \frac{\perturb_k^2}{12}   \ilbracket{ \loss^{(3)} \ilparenthesis{\overline{\btheta}_k^+} + \loss^{(3)} \ilparenthesis{\overline{\btheta}_k^{-}} } \ilparenthesis{\bDelta_k\otimes\bDelta_k\otimes\bDelta_k} \mapping_k\ilparenthesis{\bDelta_k} -  \bias_k\ilparenthesis{\hbtheta_k} , \label{eq:noise}
 	\end{align}
 	where  $\overline{\btheta}_k^{\pm}$ is some convex combination of $\hbtheta_k$ and  $ \ilparenthesis{\hbtheta_k\pm\perturb_k\bDelta_k} $.  Overall,
 	the magnitude of  the bias term  $\E _k\norm{ \bias_k\ilparenthesis{\hbtheta_k} }$    is  $ O(\perturb_k^2) $, and the second-moment of the noise term  $ \E_k\ilbracket{\norm{\noise_k}^2} $ is $ O(\perturb_k^{-2}) $.
 \end{lem} 
 The difficulty in tuning $\perturb_k$ stems  from the  trade-off between the bias term $ O(\perturb_k^2) $ and the variance term $ O(\perturb_k^{-2}) $.  
 \paragraph{Discussion on A.\ref{assume:Loss}} 
 \label{para:relaxedThrice}
 The $O(\perturb_k^2)$ bias and $O(\perturb_k^{-2})$ variance in Lemma~\ref{lem:Tradeoff} remain valid when the ``three-times continuously differentiablility'' in A.\ref{assume:Loss} is changed to ``twice-continuously differentiablility and \emph{Lipschitz} Hessian.'' Under such condition,
 we may still  obtain $ \E_k \norm{\bias_k\ilparenthesis{\hbtheta_k} }= O(\perturb_k^2) $ and $ \E_k\ilbracket{\norm{\noise_k\ilparenthesis{\hbtheta_k}}} = O(\perturb_k^{-2})  $. 
 
  \subsubsection{Almost Surely  Convergence}

 Several additional assumptions 
 are imposed to facilitate the strong convergence.

 \begin{assumeA}
 	[Iterate Boundedness and ODE Condition] \label{assume:ODE} Assume $ \norm{\hbtheta_k}\stackrel{\mathrm{a.s.}}{<}\infty $  for all $k$.  Also assume  that $ \bvartheta $ is an asymptotically stable solution of the differential equation $ \diff \bx\ilparenthesis{t}/\diff t = - \bg\ilparenthesis{\bx} $, whose solution under initial condition $\bx_0$ will be denoted as $ \bx\ilparenthesis{\given{t}{\bx_0}}$.  Moreover, let $D\ilparenthesis{\bvartheta}\equiv \ilset{\bx_0: \lim_{t\to\infty} \bx\ilparenthesis{\given{t}{\bx_0} } = \bvartheta}$. Further assume that $ \hbtheta_k $ falls within some compact subset of $D\ilparenthesis{\bvartheta}$ infinitely often for almost all sample points. 
 \end{assumeA}

 \begin{assumeAprime}{4'}[Unique Minimum]\label{assume:Convexity} 
 	Assume that $\bvartheta$ is the unique minimizer such that 
 	$ \sup\{\norm{\btheta}: \loss\ilparenthesis{\btheta} \le \loss\ilparenthesis{\bvartheta} + \constNum_1\} < \infty $ for every $\constNum_1>0$, 
 	$ \inf_{\norm{\btheta-\bvartheta}>\constNum_2}\ilbracket{\loss\ilparenthesis{\btheta}-\loss\ilparenthesis{\bvartheta}}>0 $ for every $ \constNum_2>0$,  $ \inf_{\norm{\btheta-\bvartheta}>\constNum_3} \norm{\bg\ilparenthesis{\btheta}}>0 $  for every $\constNum_3>0$.    Moreover,
 	there exists some $K$, such that for $k\ge K$, $ \bH\ilparenthesis{\cdot } $ satisfies  $\norm{ \bH\ilparenthesis{\btheta}}_{\infty} < \BoundSecondOrder $ for all $\btheta$ in an open neighborhood of $\hbtheta_k$ a.s. 
 \end{assumeAprime}

 \begin{assumeA}
 	[Stepsize] \label{assume:Stepsize} $ \gain_k>0 $, $ \perturb_k>0 $, $ \gain_k\to 0 $, $ \perturb_k\to 0 $, $ \sum_k\gain_k = \infty $, $ \sum_k\gain_k^2\perturb_k^{-2}<\infty $.\end{assumeA}
 \begin{thm}[Almost Surely Convergence]\label{thm:StrongConvergence} Under the assumptions 
 	A.\ref{assume:Loss}, A.\ref{assume:Perturbation}, A.\ref{assume:Noise} (as in Lemma~\ref{lem:Tradeoff}), along with A.\ref{assume:ODE} and A.\ref{assume:Stepsize}, we have $ \hbtheta_k \stackrel{k\to\infty}{\longrightarrow} \bvartheta $ a.s. 
 \end{thm}
 
 \begin{thmPrime}{1'}[Almost Surely Convergence] \label{thm:StrongConvergence-2}
 	Under A.\ref{assume:Loss}, A.\ref{assume:Perturbation}, A.\ref{assume:Noise}, along with A.\ref{assume:Convexity} and A.\ref{assume:Stepsize}, we have
 	\begin{enumerate}[i)]
 		\item \label{item:ub} $ \norm{\hbtheta_k}\stackrel{\mathrm{a.s.}}{<}\infty  $  for all $k$.  
 		\item \label{item:as} $ \hbtheta_k \stackrel{k\to\infty}{\longrightarrow} \bvartheta $ a.s. 
 	\end{enumerate}
 \end{thmPrime}
 
 \paragraph{Discussion  on A.\ref{assume:ODE} and A.\ref{assume:Convexity}}\label{para:ODE-Convex}
 First of all, note that neither  A.\ref{assume:ODE}   nor A.\ref{assume:Convexity} implies    the other. Moreover, $\bH\ilparenthesis{\cdot}$ being strongly convex is a \emph{sufficient} condition  for both  A.\ref{assume:ODE}   and A.\ref{assume:Convexity}. Nonetheless, strong convexity is \emph{not} a \emph{necessary} condition  for either A.\ref{assume:ODE}   and A.\ref{assume:Convexity}. Therefore, both Theorem~\ref{thm:StrongConvergence} and Theorem~\ref{thm:StrongConvergence-2} imply a.s. convergence when $\loss\ilparenthesis{\cdot}$ is strongly convex, but they also imply  the a.s. convergence result for functions that are  more complicated  beyond strongly convex functions.   
 \cite[pp. 40--41]{kushner1978stochastic} discusses why the iterate-boundedness in A.\ref{assume:ODE} \emph{may} not  not a restrictive condition and could be expected to hold in most applications.

 \subsubsection{Asymptotic Normality }
 Additional assumptions are    needed to facilitate the weak convergence result.  
 \begin{assumeA}
 	[Additional Conditions on Perturbation and Noise] \label{assume:AdditionalNormality}Assume that there exists a $\bSigma\succ \zero$ such that $ \bSigma_k\stackrel{k\to\infty}{\longrightarrow} \bSigma $. There exists some $\constNum_4>0$ such that  $ \E_k  \ilbracket{ \norm{\mapping_k\ilparenthesis{\bDelta_k}  }^{2+\constNum_4} } \stackrel{\mathrm{a.s.}}{<} \infty $ and $ \E\ilbracket{\given{\ilparenthesis{\upvarepsilon_{k}^+ - \upvarepsilon_k^-}^{2+\constNum_4}}{\hbtheta_k,\bDelta_k}}\stackrel{\mathrm{a.s.}}{<}\infty  $ uniformly for all $k$.  Finally, $\bH\ilparenthesis{\bvartheta}\succ\zero$. 
 	
 \end{assumeA}
 
 \begin{rem} 
 	\label{rem:NoiseCMC}
 	Note that under IID scenario for the observation noise, 
 	we have  $ \E\ilbracket{\given{\ilparenthesis{\upvarepsilon_k^+-\upvarepsilon_k^-}^2}{\hbtheta_k,\bDelta_k}} \to 2   \Var\ilparenthesis{ \ell\ilparenthesis{\bvartheta, \upomega} } $ a.s., where the variance is taken over $\upomega\in\Omega$.   This is  due to   $\hbtheta_k\stackrel{\mathrm{a.s.}}{\longrightarrow}\bvartheta$ shown  Theorem~\ref{thm:StrongConvergence} and $ \perturb_k\to 0 $ assumed in A.\ref{assume:Stepsize}.
 \end{rem}
 
 Let us first show the property of our Hessian estimate described in Section~\ref{subsect:decription}.
 \begin{thm}\label{thm:StrongConvergenceHessian} Under aforementioned conditions, and assume $\tperturb_k = O(\perturb_k)$, we have    $ \obH_k \stackrel{\mathrm{a.s.}}{\longrightarrow} \bH\ilparenthesis{\bvartheta} $. 
 \end{thm}
 
 We now show the rate of convergence of \algoName{HARP} in Algorithm~\ref{algo:HARP}. 
 According to  A.\ref{assume:Stepsize}, we  use  $ \gain_k = \nicefrac{\gain}{{k}^{\upalpha}} $ and $ \perturb_k = \nicefrac{\perturb}{ {k}^{\upgamma}} $ for $k\ge 0$, where 
 \begin{equation}\label{eq:Stepsize1}
 \upalpha\in\left( \nicefrac{1}{2},1\right]\,,\text{ and } \upgamma \in \ilparenthesis{0, \upalpha-\nicefrac{1}{2}}\,. 
 \end{equation}  Granted, there are other forms for stepsizes $ \ilparenthesis{\gain_k,\perturb_k} $. However,  they do not necessarily provide improved rates \cite{sacks1958asymptotic}. 
 Before stating Theorem~\ref{thm:AsymptoticNormality}, we introduce extra notations. 
 Let $\uptau =  \upalpha-2\upgamma$ and $\uptau _+ =   \uptau\cdot\indicator_{\ilset{\upalpha=1}}$.
 Let  $\bGamma_k = \gain \bH\ilparenthesis{\overline{\btheta}_k}$ with $\overline{\btheta}_k$ being some convex combination of $\hbtheta_k$ and $\bvartheta$,  $ \bm{t}_k = -\gain  k ^{\nicefrac{\uptau}{2}} \bias_k \ilparenthesis{\hbtheta_k} $, and $ \bv_k \equiv -\gain k^{-\upgamma} \noise_k\ilparenthesis{\hbtheta_k} $.  
 \begin{thm}
 	[Asymptotic Normality] \label{thm:AsymptoticNormality} 	Assume  A.\ref{assume:Loss}, A.\ref{assume:Perturbation}, A.\ref{assume:Noise},     A.\ref{assume:ODE} or A.\ref{assume:Convexity},   A.\ref{assume:Stepsize}, and  A.\ref{assume:AdditionalNormality} hold. 
 	Pick $ \gain > \nicefrac{\uptau_+}{\ilbracket{2\uplambda_{\min}\ilparenthesis{\bH\ilparenthesis{\bvartheta}}}} $ and $ \upalpha\le 6\upgamma $, we have 
 	\begin{equation}\label{eq:Normality}
 	k^{\nicefrac{\uptau}{2}} \ilparenthesis{\hbtheta_k-\bvartheta} \stackrel{\mathrm{dist.}}{\longrightarrow} \mathcal{N} \parenthesis{   \bmu,\bB }\,,
 	\end{equation} where $ (\bmu,\bB) $ satisfies  the   linear system (\ref{eq:Normality-Mean})   and the Lyapunov equation  (\ref{eq:Normality-Variance})   respectively:
 	\begin{numcases}
 	{}
 	\ilparenthesis{\bGamma - \nicefrac{\uptau_+\bI}{2}} \bmu = \bm{t}\,, \label{eq:Normality-Mean}\\
 	\ilparenthesis{\bGamma - \nicefrac{\uptau_+\bI}{2}} \bB + \bB \ilparenthesis{\bGamma ^\transpose- \nicefrac{\uptau_+\bI}{2}} = \frac{\gain^2 \Var\ilbracket{\ell\ilparenthesis{\bvartheta,\upomega}}}{2\perturb^2} \bSigma\,. 
 	\label{eq:Normality-Variance}
 	\end{numcases} 
 	In (\ref{eq:Normality-Mean}--\ref{eq:Normality-Variance}), 
 	$	\bGamma = \lim_{k\to\infty} \bGamma_k = \gain \bH\ilparenthesis{\bvartheta}$, the $ \Var\ilbracket{\ell\ilparenthesis{\bvartheta,\upomega}} $ and $ \bSigma $ are  defined  in Remark~\ref{rem:NoiseCMC} and A.\ref{assume:AdditionalNormality} respectively,   and  
 	\begin{align}\label{eq:bias1}
 	\bm{t}&= \lim_{k\to\infty}\bm{t}_k  =  -\frac{\gain\perturb^2}{6} \indicator_{\ilset{\upalpha=6\upgamma}} \E \ilbracket{ L^{(3)}\ilparenthesis{\bvartheta}\cdot  \ilparenthesis{\bDelta\otimes\bDelta\otimes\bDelta} \cdot \mapping\ilparenthesis{\bDelta}  }\,,
 	\end{align} 
 	where $ \bDelta $ is  $\zero$-mean and $\bSigma^{-1}$-covariance. 
 \end{thm}
 
 \begin{rem}\label{rem:Lyapunov}
 	\cite{bartels1972solution} provides  the explicit solution to (\ref{eq:Normality-Variance}): 
 	\begin{equation}\label{eq:Normality-Covariance}
 	\bB =\frac{\gain^2\Var\ilbracket{\ell\ilparenthesis{\bvartheta,\upomega}}}{2\perturb^2} \int_0^{\infty} e^{t\ilparenthesis{\nicefrac{\uptau_+\bI}{2} -\bGamma} } \bSigma e^{t\ilparenthesis{ \nicefrac{\uptau_+\bI}{2}  -\bGamma^\transpose}}\diff t\,. 
 	\end{equation}
 \end{rem}   \remove{\begin{rem}
 		If A.\ref{assume:Loss} is relaxed as per Subsection~\ref{para:relaxedThrice},  (\ref{eq:bias1}) becomes $ \bm{t} = \indicator_{\set{\upalpha=6\upgamma}} \cdot  O(1)  $. The matrix $\bB$  in (\ref{eq:Normality-Covariance}) remains intact. 
 \end{rem}}

 \subsection{CRN Scenario}
 \label{sect:CRN}
 This section considers the CRN noise scenario, where the   
 \emph{fastest}  rate $ O(k^{-\nicefrac{1}{2}})   $  for RMS 
 is achieved when $\upalpha=1$ and $\upgamma>\nicefrac{1}{4}$. Here, the bias-variance trade-off as arising in Lemma~\ref{lem:Tradeoff} no longer applies,  see Lemma~\ref{lem:CRN}, whence Section~\ref{sect:CRN}  has a faster convergence rate compared to Section~\ref{sect:IID}.   
 The previous assumption on the noise is now changed  for  the CRN scenario. 
 \begin{assumeAprime}{3'}
 	[CRN]  \label{assume:NoiseCRN}  $ \upomega_k (= \upomega_k^{+} = \upomega_k^-) $ are i.i.d. and are independent from $\field_k$. Let  $\noisyG\ilparenthesis{\cdot,\cdot}: \real^{\Dimension} \times\Omega\mapsto\real^\Dimension$ be the partial derivative of $ \ell\ilparenthesis{\btheta,\upomega} $ w.r.t. $\btheta$. Assume that $ \norm{\noisyG\ilparenthesis{\btheta,\upomega}}_{\infty} \le  \BoundFirstNoisy $ uniformly for all $\btheta$ and a.s. for all $\upomega$. 
 \end{assumeAprime}

 \begin{lem}[Second Moment of $\hbg_k\ilparenthesis{\hbtheta_k}$] 
 	\label{lem:CRN} When A.\ref{assume:Loss}, A.\ref{assume:Perturbation}, and A.\ref{assume:NoiseCRN} hold,   
 	\begin{align}
 	\label{eq:NoiseCRN-1}
 	\E_k \ilset{ \norm{\hbg_k\ilparenthesis{\hbtheta_k}^2} } & \stackrel{\mathrm{a.s.}}{=}  \E  \norm{\noisyG\ilparenthesis{\hbtheta_k,\upomega_k}}^2 + o(1) \stackrel{\mathrm{a.s.}}{=}  \int_{\upomega\in\Omega} \norm{\noisyG\ilparenthesis{\hbtheta_k, \upomega}}^2 \diff\Prob\ilparenthesis{\upomega} + o(1)\,. 
 	\end{align} 
 \end{lem}
 
 The a.s. convergence result is similar to Theorem~\ref{thm:StrongConvergence} or Theorem~\ref{thm:StrongConvergence-2}. The corresponding proofs are similar using Lemma~\ref{lem:CRN}. We turn to finding the convergence rate directly.  Before stating Theorem~\ref{thm:AsymptoticNormality-CRN}, we define 
 some notations. 
 Let $ \upalpha_+ \equiv \upalpha\cdot\indicator_{\set{\upalpha=1}} $.  Let $ \bGamma_k = \gain\bH\ilparenthesis{\overline{\btheta}_k} $ with $ \overline{\btheta}_k $ being some convex combination of $\hbtheta_k$ and $\bvartheta$, $ \bm{t}_k = -\gain k^{\nicefrac{\upalpha}{2}} \bias_k\ilparenthesis{\hbtheta_k} $, and $ \bv_k =-\gain \noise_k\ilparenthesis{\hbtheta_k} $.

 \begin{thm}
 	[Asymptotic Normality] \label{thm:AsymptoticNormality-CRN}  Assume  A.\ref{assume:Loss}, A.\ref{assume:Perturbation},   A.\ref{assume:NoiseCRN},     A.\ref{assume:ODE} or A.\ref{assume:Convexity}, A.\ref{assume:Stepsize}, A.\ref{assume:AdditionalNormality}. Pick $ \gain > \nicefrac{\upalpha_+}{\ilbracket{2\uplambda_{\min}\ilparenthesis{\bH\ilparenthesis{\bvartheta}} } } $ and $ \upalpha <  4\upgamma $, we have 
 	\begin{equation}\label{eq:asymNormalCRN}
 	k^{\nicefrac{\upalpha}{2}} \ilparenthesis{\hbtheta_k-\bvartheta}  \stackrel{\mathrm{dist.}}{\longrightarrow}  \mathcal{N}\ilparenthesis{\zero, \bB}
 	\,,\end{equation}
 	where $\bB$ satisfies 
 	\begin{equation}\label{eq:CovarianceDiscussion3}
 	\ilparenthesis{\bGamma - \nicefrac{\upalpha_+\bI}{2}} \bB + \bB \ilparenthesis{\bGamma^\transpose-\nicefrac{\upalpha_+ \bI}{2}} = \gain^2 \noisyGcov\,.
 	\end{equation} Here, $\bGamma = \lim_{k\to\infty} \bGamma_k = \gain\bH\ilparenthesis{\bvartheta}$, 	and $ \noisyGcov $ has elements  	
 	\begin{align}
 	\label{eq:noisyGcov}
 	&  \noisyGcov_{i,j} = \indicator_{\set{i=j}}  \int_{\upomega\in\Omega} \norm{\noisyG\ilparenthesis{\bvartheta,\upomega}}^2\diff\Prob\ilparenthesis{\upomega}  + \indicator_{\set{i\neq j }}  \int_{\upomega\in \Omega } \ilbracket{\noisyGcomponent\ilparenthesis{\bvartheta,\upomega}}_{i}  \ilbracket{\noisyGcomponent\ilparenthesis{\bvartheta,\upomega}}_{j}  \diff \Prob\ilparenthesis{\upomega} \,,
 	\end{align} where $ \ilbracket{\noisyGcomponent\ilparenthesis{\bvartheta,\upomega}}_{i}$ denotes the $i$th component of $   {\noisyGcomponent\ilparenthesis{\bvartheta,\upomega}} $.
 \end{thm}

 Recall that in  IID scenario, (\ref{eq:Normality})  involves a nonzero  $\bmu$ 
 when the fastest rate $ O(k^{-\nicefrac{1}{3}}) $ is achieved at $(\upalpha,\upgamma) = (1,\nicefrac{1}{6})$. 
 On the contrary, in the CRN scenario,   the mean in  (\ref{eq:asymNormalCRN})  
 is zero when the fastest rate $ O(k^{-\nicefrac{1}{2}}) $ is achieved whenever $ \ilparenthesis{\upalpha,\upgamma}  = \ilparenthesis{1, >\nicefrac{1}{4}} $. 
 \begin{rem} 
 	The asymptotic result shows that the covariance structure $\bSigma_k(\to\bSigma)$ for $\bDelta_k$ no longer impacts the asymptotic normality (rate of convergence). Instead, the moments of $ \noisyG\ilparenthesis{\bvartheta,\upomega} $ takes over given the assumed differentiablility of  the random function $ \ell\ilparenthesis{\btheta,\upomega} $ in A.\ref{assume:NoiseCRN}. 
 \end{rem}

\subsection{ Comparison Between \algoName{HARP} and  \algoName{SPSA}} \label{subsect:HARP}   Let us see what happens when $\bSigma_k \to \bSigma = \bH\ilparenthesis{\bvartheta}$.  
Let us write out (\ref{eq:Normality-Covariance}) in  Remark~\ref{rem:Lyapunov} for $\upalpha<6\upgamma$.
Let the eigen-decomposition of $\bH\ilparenthesis{\bvartheta}$ be $\bP\bLambda \bP^\transpose$, for orthogonal matrix $\bP$ and diagonal matrix $\bLambda = \diag\ilparenthesis{\uplambda_1,\cdots,\uplambda_{\Dimension}}$. Then $\bB$ in (\ref{eq:Normality-Variance}) equals  $ \bP\bM\bP^\transpose $, where the $ (i,j) $th  elements of $\bM$ is 
\begin{equation*}
m_{i,j} =    \frac{\gain^2\Var\ilparenthesis{ \ell\ilparenthesis{\bvartheta, \upomega} }}{2\perturb^2} \ilparenthesis{\bP^\transpose\bSigma\bP}_{i,j}\ilparenthesis{\gain\uplambda_i+\gain\uplambda_j-\uptau_+}^{-1}\,.
\end{equation*}  
For all the algorithms listed in Subsection~\ref{subsect:PriorWork}, with  $\bSigma_k=\bI$,  the trace of the covariance term is asymptotic to 
\begin{equation}\label{eq:trace-1SPSA}
\frac{\gain^2\Var\ilbracket{\ell\ilparenthesis{\bvartheta,\upomega}}}{2\perturb^2} \sum_{i=1}^{\Dimension} \ilparenthesis{2\gain \uplambda_i - \uptau_+}^{-1}\,,
\end{equation} whereas \algoName{HARP}   in Algorithm~\ref{algo:HARP}, with  $\bSigma_k=\hbH_k  \to \bH\ilparenthesis{\bvartheta}$, gives
\begin{equation}\label{eq:trace-1HARP}
\frac{\gain^2\Var\ilbracket{\ell\ilparenthesis{\bvartheta,\upomega}}}{2\perturb^2}  \sum_{i=1}^{\Dimension} \frac{1}{2\gain - \nicefrac{\uptau_+}{\uplambda_i}}\,.
\end{equation} 
Note that both (\ref{eq:trace-1SPSA}) and (\ref{eq:trace-1HARP}) diverge when \emph{any} one of the eigenvalues of $\bH\ilparenthesis{\bvartheta}$ is close to zero.
Nonetheless, (\ref{eq:trace-1HARP}) is smaller than (\ref{eq:trace-1SPSA}) when $ \uplambda_i\ll 1 $ for some $1\le i \le \Dimension$, under which circumstance the iteration complexity (\ref{eq:iterationComplexity})  of \algoName{HARP} \emph{can} be better than that of \algoName{SPSA}\textemdash at the cost of  two additional ZO queries per iteration, see the last line in Algorithm~\ref{algo:HARP}.

 \section{Numerical Illustration} \label{sect:Numerical}
    We now present two  empirical examples to demonstrate the fast optimization and  the wide applicability of \algoName{HARP}.

 \subsection{Synthetic Problem: Skew-Quartic Function}

  Section~\ref{subsect:HARP}  demonstrates that \algoName{HARP} performs better under ill-conditioned problem.
  This synthetic example uses  the skew-quartic function in  \cite{spall2000adaptive} as the true loss $\loss\ilparenthesis{\cdot}$  in (\ref{eq:SOsetup}). The corresponding  Hessian has one single large eigenvalue and $ (\Dimension-1) $ close-to-zero eigenvalues. This loss function is poorly-conditioned. 
  The noisy loss observation $ \ell \ilparenthesis{\btheta,\upomega}  $ in (\ref{eq:SOsetup}) is the true loss   corrupted by an i.i.d. $\mathcal{N}\ilparenthesis{0,1}$ random noise. We use  $\Dimension=20$ and initialize $\hbtheta_0$     within $ \ilbracket{-20, 20}^{\Dimension} $.  We use 
  $\gain_k = \nicefrac{\gain}{\ilparenthesis{k+1 + A}^{\upalpha}}$ with
  $\upalpha=0.602$ and $A $ equals $10\%$ of the iteration number, $\perturb_k = \nicefrac{\perturb}{\ilparenthesis{k+1}^{\upgamma}}$ with $\upgamma=0.101$. Number of replicates is $25$ (i.e., all the plots below are averaged performance over $25$ replications). The corresponding implementation details ca be found at \href{https://github.com/jingyi-zhu/Fast2SPSA/tree/master/HARP}{GitHub}.   The algorithm we compare against is \algoName{SPSA} \cite{spall1992multivariate}, which has comparable/better performance than other algorithms reviewed in Section~\ref{subsect:PriorWork}.  During the implementation, both \algoName{SPSA} and \algoName{HARP}   use exactly four ZO queries each iteration, so  the query complexity \emph{aligns} with the iteration complexity. We see from Figure~\ref{fig:Skew-Quartic} that that \algoName{HARP} with $\bSigma_k = \hbH_k $ outperforms \algoName{SPSA} with $\bSigma_k = \bI$ for the ill-conditioned problem of minimizing a skew-quartic function. 
  
  \begin{figure}[!htbp]
  	\centering
  	% include second image
  	\includegraphics[width=.5\linewidth]{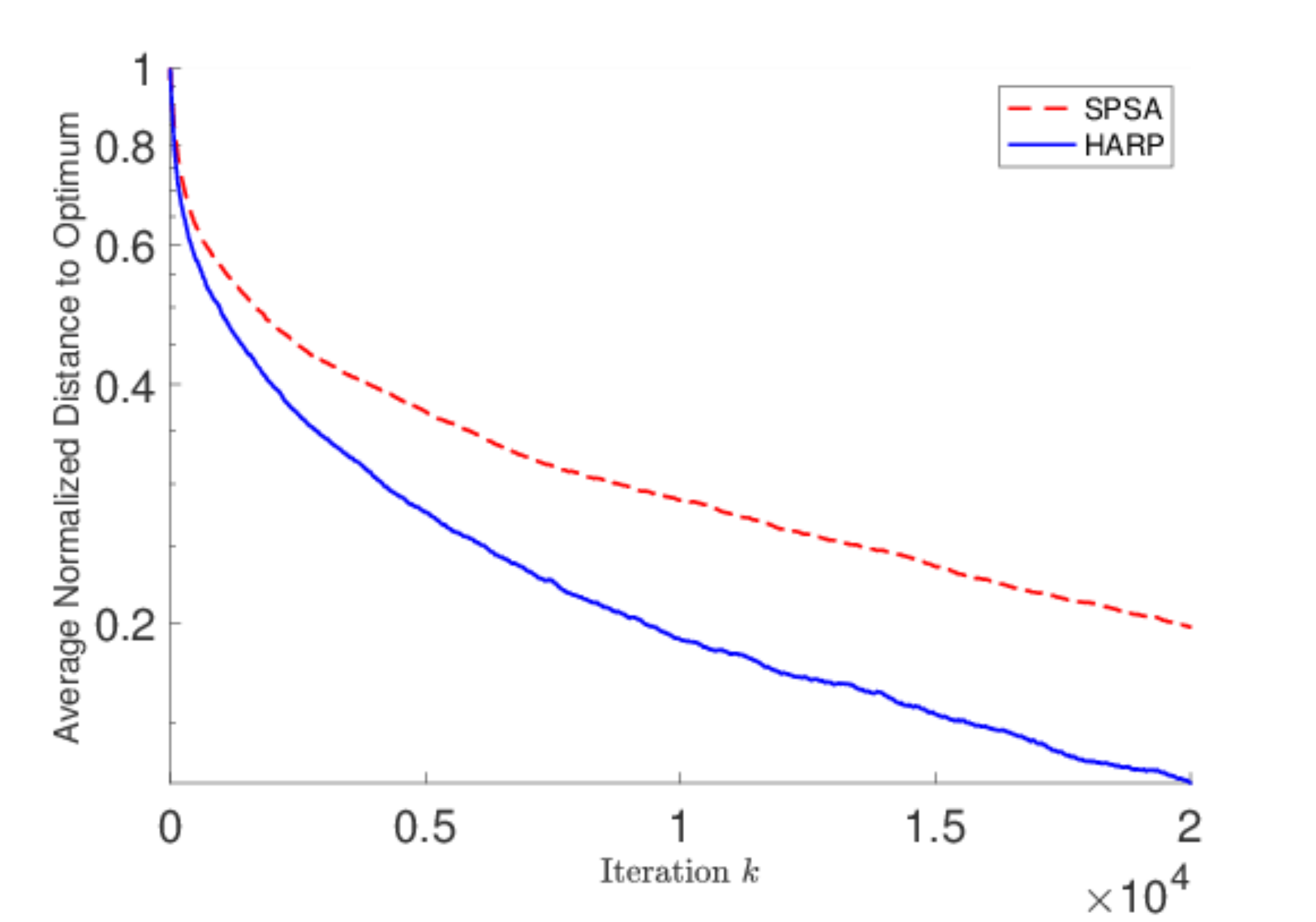} 
  	\caption{ Performance of \algoName{SPSA}  and \algoName{HARP}   in terms of normalized 
  		distance $ \nicefrac{\norm{\hbtheta_k-\bvartheta}}{\norm{\hbtheta_0-\bvartheta}} $   averaged across  $25$ independent replicates, and  both algorithms use four ZO queries per iteration.  The underlying loss function  is the skew-quartic function with   $\Dimension=20$, and the noisy observation is corrupted by a $\mathcal{N}\ilparenthesis{0,1}$  noise. 
  	}
  	\label{fig:Skew-Quartic}
  \end{figure}

 \subsection{Universal Image Attack As A Finite-Sum Problem} 
 
 We consider the problem of generating black-box adversarial examples universally for $I>1$ images \cite{chen2017zoo,cheng2018query} using zeroth-order optimization methods.     We consider  the constrained problem 
 \begin{equation}
 \begin{cases}
 &\min_{\btheta}   \loss\ilparenthesis{\btheta }\equiv  \underbrace{\penaltyPara \norm{\btheta}_2^2}_{\equiv \loss_1\ilparenthesis{\btheta}}+   \underbrace{\frac{1}{I}\sum_{i=1}^I \individualLoss\ilparenthesis{\bzeta_i+\btheta} }_{\equiv \loss_2\ilparenthesis{\btheta}}\,, \\
 &\text{ s.t. } \ilparenthesis{\bzeta_i+\btheta}\in\bracket{-0.5, 0.5}^{\Dimension}, \forall i,   \,\, 
 \end{cases}  \label{eq:PerImage2}
 \end{equation} where the constraint is to normalize the resulting pixels within the range $ \bracket{-0.5, 0.5}^{\Dimension} $. The  $ \individualLoss\ilparenthesis{\cdot}:\real^\Dimension\mapsto\real $ imposed 
on each image takes the form
\begin{equation}
\individualLoss\ilparenthesis{\bzeta} = \max_{i:1\le i\le C} \set{ \mathrm{ps}\ilparenthesis{\bzeta, i } - \max_{j\neq i: 1\le j \le C} \bracket{\mathrm{ps}\ilparenthesis{\bzeta,j}} }\,,
\end{equation}
where $ \mathrm{ps}\ilparenthesis{\bzeta,i} $ denotes the prediction score of the $i$-th class given the input $\bzeta$. The model $ \mathrm{ps}\ilparenthesis{\cdot,\cdot} $here is trained using the structure specified   in \cite{carlini2017towards}. Note that $ \sum_{i=1}^I \individualLoss\ilparenthesis{\bzeta_i+\btheta} = 0 $ when the chosen  images $ \ilset{ \bzeta_i}_{i=1}^I $ are successfully attacked by the universal perturbation $\btheta$. 
The noisy loss observation $ \ell\ilparenthesis{\btheta,\upomega} $ is 
\begin{equation}
\ell\ilparenthesis{\btheta,\upomega} =  \penaltyPara\norm{\btheta}_2^2+ \frac{1}{J} \sum_{j=1}^{J}    \individualLoss\ilparenthesis{\bzeta_{i_j\ilparenthesis{\upomega}} +\btheta}  \,,
\end{equation} for   $  J\le I$, and the $J$      indexes  $\ilset{i_1\ilparenthesis{\upomega},\cdots, i_J\ilparenthesis{\upomega}}$ 
are i.i.d. uniformly drawn from $ \ilset{1,\cdots,I} $ (without replacement). 

  Consider (\ref{eq:PerImage2}) with  $\penaltyPara=\nicefrac{1}{10}$. The $I$ images arising in  (\ref{eq:PerImage2}) are those  \emph{correctly} classified by the trained model.    $ \Dimension = 784 $ for MNIST dataset. 
 The algorithm we compare against is \algoName{ZO-AdaMM} \cite{chen2019zo}.  Both algorithms are  initialized at      $ \hbtheta_0=\zero $. The ZO-query per iteration for both algorithms is 60, so the query complexity \emph{aligns} with the iteration complexity.     We perform 25 independent  replicates, each with $K=1000$ iterations.     The stepsizes are $\gain_k = \nicefrac{\gain}{(k+1+A)^{0.602}}$ and $\perturb_k = \nicefrac{\perturb}{(k+1)^{0.101}}$. The details of the hyper-parameters are in  \href{https://github.com/jingyi-zhu/HARP}{GitHub}.

 \begin{table}[!htbp]
 	\centering
 	\begin{tabular}{ c|| c |  c | c   } 
 		\hline 
 		Algo   &   $\E \ilbracket{\loss\ilparenthesis{\hbtheta_{K}}}$& $\ilset{\Var \ilbracket{\loss\ilparenthesis{\hbtheta_{K}}}}^{\frac{1}{2}}$  & $ \E \ilbracket{\loss_2\ilparenthesis{\hbtheta_K}}$    \\
 		\hline \hline
 		\multirow{1}{*}{ \shortstack[l]{\algoName{AdaMM}}  }  
 		& $185.96$ & $16.88$ & $40.95$  \\
 		\hline
 		\multirow{1}{*}{\algoName{HARP}}   
 		& $138.22$ & $18$ & $12.50$   \\
 		\hline \hline 
 	\end{tabular}\caption{
 		Performance of \algoName{ZO-AdaMM} and \algoName{HARP} in terms of loss   after $K=1000$ iterations  averaged across $25$ independent replicates. The loss function $L\ilparenthesis{\cdot}$  is the sum of  the magnitude cost $L_1\ilparenthesis{\cdot}$ and  the  attack loss $L_2\ilparenthesis{\cdot}$.  Here $L_2\ilparenthesis{\cdot}$   measures the  attack loss on   $ I= 100$ images of the letter one, and its \emph{noisy} query is evaluated using a batch-size of one. A close-to-zero $L_2\ilparenthesis{\cdot}$ loss is equivalent to a close-to-one attack success rate. 
 		\remove{
 			AdaMM: a = 0.5 (0.6), c = 0.04.   
 		}
 		\remove{
 			HARP: a = 0.005, c = 0.04, w = 1e-5. 
 		}      
 	} \label{tab:universalOn1}
 \end{table} 
 
 \remove{ 
 	\begin{table}[!htbp]
 		\centering
 		\begin{tabular}{ c || c | c | c   } 
 			\hline 
 			Algo   &   $\E \ilbracket{\loss\ilparenthesis{\hbtheta_{K}}}$& $\ilset{\Var \ilbracket{\loss\ilparenthesis{\hbtheta_{K}}}}^{\frac{1}{2}} $& $ \E \ilbracket{\loss_2\ilparenthesis{\hbtheta_K}}$   \\
 			\hline \hline
 			\multirow{1}{*}{ \shortstack[l]{\algoName{AdaMM}}  }  
 			& $236.61$ & $10.92$ & $80.86
 			$   \\
 			\hline
 			\multirow{1}{*}{\algoName{HARP}}   
 			& $189.73$ & $6.06$ & $60.23$   \\
 			\hline \hline 
 		\end{tabular}\caption{Attacking
 			images of ``7'' with  $ (I,J) = (50,10) $.  
 			\remove{
 				AdaMM: a = 5, c = 0.04.  
 			}
 			\remove{
 				HARP: a = 0.08, c = 0.08, w = 1e-5. 
 			}
 		} \label{tab:universalOn7}
 	\end{table}
 	
 }

 \begin{table}[!htbp]
 	\centering
 	\begin{tabular}{ c|| c | c | c   } 
 		\hline 
 		Algo   &   $\E \ilbracket{\loss\ilparenthesis{\hbtheta_{K}}}$& $\ilset{\Var \ilbracket{\loss\ilparenthesis{\hbtheta_{K}}}}^{\frac{1}{2}}$ & $ \E \ilbracket{\loss_2\ilparenthesis{\hbtheta_K}}$   \\
 		\hline \hline
 		\multirow{1}{*}{ \shortstack[l]{\algoName{AdaMM}}  }  
 		& $56.95$ & $6.89$ & $11.75$   \\
 		\hline
 		\multirow{1}{*}{\algoName{HARP}}   
 		& $18.46$ & $1.37$ & $0.13$   \\
 		\hline \hline 
 	\end{tabular}\caption{Here $L_2\ilparenthesis{\cdot}$   measures the  attack loss on   $ I= 10$ images of the letter three, and its ZO query is noise-free. 
 		\remove{
 			AdaMM: a = 1, c = 0.04.  
 		}    
 		\remove{ 
 			HARP: a = 0.1, c = 0.04, w = 1e-5. 
 		}
 	} \label{tab:universalOn3}
 \end{table} 
 Tables~\ref{tab:universalOn1}\textendash \ref{tab:universalOn3}   summarize the 
 terminal expected  loss function evaluation $ \E\ilbracket{\loss\ilparenthesis{\hbtheta_{1000}}} $, the terminal standard deviation of loss function evaluation $ \ilset{\Var\ilbracket{\loss\ilparenthesis{\hbtheta_{1000}}}}^{\frac{1}{2}} $, 
 and the terminal attack loss $\E\ilbracket{\loss_2\ilparenthesis{\hbtheta_{1000}}}$, all three of which are  averaged across $25$ independent replicates. 
 The  terminal scaled magnitude of the perturbation $ \E\ilbracket{\loss_1\ilparenthesis{\hbtheta_{1000}}} $  can be computed as (\ref{eq:PerImage2}).  
 Moreover, the noisier the loss function observation is (i.e., the larger the discrepancy between collected sample size $J$ and the entire sample size $I$), the more difficult it is to reduce the $ \E\ilbracket{\loss\ilparenthesis{\hbtheta_{K}}} $ given a fixed iteration/query budget. In both noisy and noise-free ZO queries, \algoName{HARP} shows faster convergence rate than  \algoName{ZO-AdaMM} given a fixed query-budget.

\section{Concluding Remarks}\label{sect:Conclude}
This work proposes \algoName{HARP} to use the second-order approximation from  ZO queries in both the random perturbation and the parameter update, and demonstrates its superiority in ill-conditioned problems theoretically in Section~\ref{subsect:HARP}   and numerically in Section~\ref{sect:Numerical}. Note that all the prior work use an identity/scalar matrix as the covariance matrix for the perturbation $\bDelta_k$ and use a \emph{deterministic} mapping $\mapping_k\ilparenthesis{\cdot}$. This work shows the benefits of using non-identity matrix as the covariance matrix for $\bDelta_k$ and a \emph{stochastic} mapping $\mapping_k\ilparenthesis{\cdot}$ which is $  {\field_k}$-measurable. This generalization allows experimenters to incorporate various self-learning structure on the random directions $\bDelta_k$.

Some potential future work includes (1) the generalization to root-finding problem where the Jacobian matrix is possibly asymmetric\footnote{Note that in our discussion, the Hessian matrix for minimization problem is symmetric.}; (2) the generalization to the one-measurement counterpart to (\ref{eq:ZOgradientTwoMeasurements})  as \cite{spall1997one} to further reduce query complexity; (3) the extended discussion on global convergence in line of  \cite{maryak2001global}; (4) the extension to constrained minimization problems, and the follow-up discussion when sparsity-promoted constraints are imposed; (5) the potential exploration on (early) stopping SA iterations based on the root-mean-squared error; (6) other forms  of $\bSigma_k$, including diagonal forms to reduce floating point operations per iteration.

 \section*{Acknowledgment}
 The author would like to thank Dr. Zhenliang Zhang, Dr. Jian Tan, and Dr. Wotao Yin for inspirational discussion.

\bibliography{sample}

 \clearpage
\appendix
 
 \section{Supplementary  Proofs}

\begin{proof}
	[Proof for Lemma~\ref{lem:Tradeoff}] First consider the bias term $ \bias_k\ilparenthesis{\hbtheta_k} $ of $ \hbg_k\ilparenthesis{\hbtheta_k} $ as an estimator for $\bg\ilparenthesis{\hbtheta_k}$. 
	\begin{align}
	& \E _k\ilbracket{\hbg_k\ilparenthesis{\hbtheta_k}} \nonumber\\
	&\,\, \stackrel{\text{a.s.}}{=}\E _k \bracket{ \frac{ \loss\ilparenthesis{\hbtheta_k+\perturb_k\bDelta_k} - \loss\ilparenthesis{\hbtheta_k-\perturb_k\bDelta_k} }{2\perturb_k} \mapping_k\ilparenthesis{\bDelta_k} }  +  \E _k \bracket{\frac{\mapping_k\ilparenthesis{\bDelta_k}}{2\perturb_k} \E\ilbracket{\given{\ilparenthesis{\upvarepsilon_k^+ - \upvarepsilon_k^-}}{\hbtheta_k,\bDelta_k}}}  \label{eq:Tradeoff1} \\ 
	&\,\,\stackrel{\text{a.s.}}{=} \E _ k \ilbracket{\mapping_k\ilparenthesis{\bDelta_k} \bDelta_k^\transpose } \bg\ilparenthesis{\hbtheta_k}  + \frac{\perturb_k^2}{12} \E _k \set{ \ilbracket{ \loss^{(3)} \ilparenthesis{\overline{\btheta}_k^+} + \loss^{(3)} \ilparenthesis{\overline{\btheta}_k^{-}} } \ilparenthesis{\bDelta_k\otimes\bDelta_k\otimes\bDelta_k} \mapping_k\ilparenthesis{\bDelta_k}  } \label{eq:Tradeoff2}
	\\
	&\,\,   \stackrel{\text{a.s.}}{=} \bg\ilparenthesis{\hbtheta_k} + \bias_k\ilparenthesis{\hbtheta_k}\,, \label{eq:Tradeoff2-2}
	\end{align}
	where equation  (\ref{eq:Tradeoff1}) uses \cite[Thm. 9.1.3 on p. 315]{chung2001course}, equation (\ref{eq:Tradeoff2}) 
	uses the third-order  Taylor expansion with mean-value forms of the remainder and 
	$ \E\ilbracket{\given{\upvarepsilon_k^+-\upvarepsilon_k^-}{\hbtheta_k,\bDelta_k}} \stackrel{\text{a.s.}}{=} 0  $ in A.\ref{assume:Noise}, equation (\ref{eq:Tradeoff2-2}) uses the expression (\ref{eq:bias})   and $ \E_k \ilbracket{\mapping_k\ilparenthesis{\bDelta_k} \bDelta_k^\transpose}    \stackrel{\text{a.s.}}{=} \bI $ assumed in A.\ref{assume:Perturbation}.  Then  
	\begin{align}
	&\E_k\ilbracket{ \norm{\bias_k\ilparenthesis{\hbtheta_k}} } \nonumber\\
	&\,\,\stackrel{\text{a.s.}}{\le }     \frac{\perturb_k^2}{6} \norm{\loss^{(3)}\ilparenthesis{\btheta}}_{\infty} \E _k \ilbracket{ \norm{\bDelta_k\otimes\bDelta_k\otimes\bDelta_k} \norm{\mapping_k\ilparenthesis{\bDelta_k} }} \label{eq:Tradeoff3}\\
	&\,\, \stackrel{\text{a.s.}}{=}  \frac{\perturb_k^2}{6} \BoundThirdOrder \E_k \ilbracket{\norm{\bDelta_k}^3 \mapping_k\ilparenthesis{\bDelta_k}}   \label{eq:Tradeoff4}\\
	&\,\, \stackrel{\text{a.s.}}{\le }  \frac{\perturb_k^2}{6}\BoundThirdOrder\BoundPerturbation\,, \label{eq:Tradeoff4-2}
	\end{align}
	where inequality  (\ref{eq:Tradeoff3}) uses the mean-value theorem ($ \int_D \abs{f_1\ilparenthesis{x}f_2\ilparenthesis{x}}\diff x \le \sup_{x\in D} \abs{f_1\ilparenthesis{x}} \int _D \abs{f_2\ilparenthesis{x}}\diff x $ for two functions $ f_1 $ and $ f_2 $ and some domain of integration $D$),  equality (\ref{eq:Tradeoff4}) uses the independence between $\hbtheta_k$ and $\bDelta_k$ assumed in A.\ref{assume:Perturbation} and \cite{lancaster1972norms}, and  inequality (\ref{eq:Tradeoff4-2}) uses A.\ref{assume:Perturbation}.
	The representation of $\noise_k\ilparenthesis{\hbtheta_k}$ in (\ref{eq:noise}) follows directly from (\ref{eq:gDecomposition}) and (\ref{eq:bias}).

	We then  consider the second-moment of $\noise_k\ilparenthesis{\hbtheta_k}$ through the following computation:
	\begin{align}
	&   \E_k \set{\norm{\hbg_k\ilparenthesis{\hbtheta_k}}^2} \nonumber\\
	&\,\, \stackrel{\text{a.s.}}{=}   \E_k \set{\largenorm{\frac{\loss\ilparenthesis{\hbtheta_k+\perturb_k\bDelta_k} - \loss\ilparenthesis{\hbtheta_k-\perturb_k\bDelta_k}}{2\perturb_k} \mapping_k\ilparenthesis{\bDelta_k} }^2}  \label{eq:Tradeoff7}\\
	&\,\, \quad +   \frac{1}{4\perturb_k^2}  \E_k \ilbracket{ \ilparenthesis{\upvarepsilon_k^+-\upvarepsilon_k^-}^2 \norm{\mapping_k\ilparenthesis{\bDelta_k}}^2 } \label{eq:Tradeoff6}\\
	&\,\,  \quad+
	\frac{1}{2\perturb_k^2}  \E _k  \bigg\{  \ilbracket{ \loss\ilparenthesis{\hbtheta_k+\perturb_k\bDelta_k} - \loss\ilparenthesis{\hbtheta_k-\perturb_k\bDelta_k} }   \ilparenthesis{\upvarepsilon_k^+ -\upvarepsilon_k^-} \norm{\mapping_k\ilparenthesis{\bDelta_k}}^2  \bigg\} . \label{eq:Tradeoff5}
	\end{align}
	The term on (\ref{eq:Tradeoff6}) becomes $ O(\perturb_k^{-2}) $ because 
	\begin{align}
	&\E_k\ilbracket{\ilparenthesis{\upvarepsilon_k^+-\upvarepsilon_k^-}^2\norm{\mapping_k\ilparenthesis{\bDelta_k}}^2} \nonumber\\
	&\,\,\stackrel{\text{a.s.}}{=} \E_k \bracket{ \norm{\mapping_k\ilparenthesis{\bDelta_k}}^2 \E \ilbracket{\given{\ilparenthesis{\upvarepsilon_k^+-\upvarepsilon_k^-}^2}{\hbtheta_k,\bDelta_k}} } \label{eq:Tradeoff8}\\
	&\,\,\stackrel{\text{a.s.}}{=} \BoundNoise \cdot \E_k \ilbracket{\norm{\mapping_k\ilparenthesis{\bDelta_k}}^2} \label{eq:Tradeoff9} \\
	&\,\,\stackrel{\text{a.s.}}{\le } \BoundNoise\BoundPerturbation\,, \label{eq:Tradeoff10}
	\end{align}
	where inequality  (\ref{eq:Tradeoff8}) uses \cite[Thm. 9.1.3]{chung2001course}, inequality (\ref{eq:Tradeoff9}) uses A.\ref{assume:Noise} and the independence between $\hbtheta_k$ and $\bDelta_k$, and inequality (\ref{eq:Tradeoff10}) uses  A.\ref{assume:Perturbation}. 
	The term on (\ref{eq:Tradeoff5}) becomes zero thanks to \cite[Thm. 9.1.3]{chung2001course} and $ \E\ilbracket{\given{\upvarepsilon_k^+-\upvarepsilon_k^-}{\hbtheta_k,\bDelta_k}} $
	assumed in A.\ref{assume:Noise}. The term on (\ref{eq:Tradeoff7}) can be bounded from above by   $ \BoundPerturbation \norm{\bg\ilparenthesis{\hbtheta_k}}^2 + O(\perturb_k^2) $, as
	\begin{align}
	&   \E_k \set{\largenorm{\frac{\loss\ilparenthesis{\hbtheta_k+\perturb_k\bDelta_k} - \loss\ilparenthesis{\hbtheta_k-\perturb_k\bDelta_k}}{2\perturb_k} \mapping_k\ilparenthesis{\bDelta_k} }^2}\nonumber\\
	&\,\, \stackrel{\text{a.s.}}{=}   \ilbracket{\bg\ilparenthesis{\hbtheta_k}}^\transpose \E_k \ilset{ \bDelta_k\ilbracket{\mapping_k\ilparenthesis{\bDelta_k}}^\transpose \mapping_k\ilparenthesis{\bDelta_k}\bDelta_k^\transpose } \bg\ilparenthesis{\hbtheta_k}\nonumber\\
	&\,\, \,\, + \frac{\perturb_k^4}{144} \E_k   \largenorm{ \ilbracket{ \loss^{(3)} \ilparenthesis{\overline{\btheta}_{k}^{+}}  +  \loss^{(3)} \ilparenthesis{\overline{\btheta}_k^{-}} }  \ilparenthesis{\bDelta_k\otimes\bDelta_k\otimes\bDelta_k} \mapping_k\ilparenthesis{\bDelta_k}     } ^2 \nonumber\\
	&\,\, \,\, + \frac{\perturb_k^2}{6} \ilbracket{\bg\ilparenthesis{\hbtheta_k}}^\transpose \E _k  \bigg\{  \bDelta_k\ilbracket{\mapping_k\ilparenthesis{\bDelta_k}}^\transpose\ilbracket{ \loss^{(3)} \ilparenthesis{\overline{\btheta}_{k}^{+}}  +  \loss^{(3)} \ilparenthesis{\overline{\btheta}_k^{-}} }  \times    \ilparenthesis{\bDelta_k\otimes\bDelta_k\otimes\bDelta_k} \mapping_k\ilparenthesis{\bDelta_k}  \bigg\} \nonumber\\
	&\,\,\stackrel{\text{a.s.}}{=} O\parenthesis{\norm{\bg\ilparenthesis{\hbtheta_k}}^2} + O\ilparenthesis{\perturb_k^2}\,, 
	\end{align} 
	thanks to A.\ref{assume:Perturbation} and third-order Taylor expansion. 
\end{proof}

\begin{proof}
	[Illustration for Paragraph~\ref{para:relaxedThrice}]   The proof directly follows from  the second-order Taylor expansion and   the Lipschitz Hessian condition on the remainder terms.  
	\begin{align}
	&\E_k\ilbracket{\hbg_k\ilparenthesis{\hbtheta_k}}\nonumber\\
	&\,\, \stackrel{\text{a.s.}}{=} \E_k \bracket{ \frac{\loss\ilparenthesis{\hbtheta_k+\perturb_k\bDelta_k} - \loss\ilparenthesis{\hbtheta_k-\perturb_k\bDelta_k}}{2\perturb_k}  \mapping_k\ilparenthesis{\bDelta_k}}  + \E _k \bracket{\frac{\mapping_k\ilparenthesis{\bDelta_k}}{2\perturb_k} \E\bracket{\given{\ilparenthesis{\upvarepsilon_k^+-\upvarepsilon_k^-}}{\hbtheta_k,\bDelta_k}}} \nonumber\\
	&\,\, \stackrel{\text{a.s.}}{=}  \E_k\ilbracket{\mapping_k\ilparenthesis{\bDelta_k}\bDelta_k^\transpose}\bg\ilparenthesis{\hbtheta_k}  + \frac{\perturb_k}{4} \E_k \set{ \bDelta_k^\transpose \ilbracket{\bH\ilparenthesis{\overline{\btheta}_k^+ } - \bH\ilparenthesis{\overline{\btheta}_k^-}} \bDelta_k } \label{eq:Hlipz0}\\
	&\,\,\stackrel{\text{a.s.}}{=}  \bg\ilparenthesis{\hbtheta_k} + \bias_k\ilparenthesis{\hbtheta_k}\,, \nonumber
	\end{align} where (\ref{eq:Hlipz0}) follows from the second-order Taylor expansion.
	Then  $\bias_k\ilparenthesis{\hbtheta_k}$ satisfies
	\begin{align}
	  \E_k\norm{\bias_k\ilparenthesis{\hbtheta_k}}   & \stackrel{\text{a.s.}}{\le}   \frac{\perturb_k}{4} \E _k  \set{\bDelta_k^\transpose \bracket{O(1) \norm{2\perturb_k\bDelta_k}} \bDelta_k }\label{eq:Hlips} \\
	&\,\, \stackrel{\text{a.s.}}{=} O(\perturb_k^2) \label{eq:Hlips1}
	\end{align} where the $O(1)$ in (\ref{eq:Hlips}) represents the Lipschitz parameter of $\bH\ilparenthesis{\cdot}$.  Note that the explicit scaling constant in (\ref{eq:Hlips1}) is no longer available as (\ref{eq:bias}). 
\end{proof}

\begin{proof}[Proof for Theorem~\ref{thm:StrongConvergence}]
	
	Under  assumptions  A.\ref{assume:ODE}, and A.\ref{assume:Stepsize}, we known from \cite[Thm. 2.3.1 on p. 39]{kushner1978stochastic} that Thm.~\ref{thm:StrongConvergence} holds  when the following two conditions  hold:
	\begin{enumerate}[i)]
		\item \label{item:as1} $\norm{ \bias_k\ilparenthesis{\hbtheta_k}}<\infty $ for all $k$ and 
		$ \bias_k\ilparenthesis{\hbtheta_k}\to \zero  $ a.s. 
		\item \label{item:as2} $\lim_{k\to\infty} \Prob\set{ \sup_{j\ge k} \norm{ \sum_{i=k}^j \gain_i\noise_i\ilparenthesis{\hbtheta_k} } \ge \upeta } = 0$ for any $\upeta>0$. 
	\end{enumerate}
	Obviously, \ref{item:as1} holds thanks to Lemma~\ref{lem:Tradeoff}. Under assumption  A.\ref{assume:Noise}, $\noise_k\ilparenthesis{\hbtheta_k} $ defined in (\ref{eq:noise}) is an $ \field_k $-martingale. Using \cite[Eq. (4.1.4)]{kushner2003stochastic}, we have
	\begin{align}
	&\Prob\set{\sup_{j\ge k}\norm{\sum_{i=k}^j \gain_i\noise_i\ilparenthesis{\hbtheta_i}}\nonumber	\ge \upeta} \\
	&\,\, \le \upeta^{-2} \E \norm{ \sum_{i=k}^{\infty} \gain_i\noise_i\ilparenthesis{\hbtheta_i} }^2\label{eq:StrongConvergence0}\\
	&\,\,= \upeta^{-2} \sum_{i=k}^{\infty} \gain_i^2 \E \norm{\noise_i\ilparenthesis{\hbtheta_i}}^2\,, \label{eq:StrongConvergence1}
	\end{align} 
	where inequality (\ref{eq:StrongConvergence0}) uses Markov's inequality, equality   (\ref{eq:StrongConvergence1}) uses $ \E[ ]\noise_i\ilparenthesis{\hbtheta_i}^\transpose\noise_j\ilparenthesis{\hbtheta_j} ]= \E \{\noise_i\ilparenthesis{\hbtheta_i}^\transpose \E \ilbracket{\given{\noise_j\ilparenthesis{\hbtheta_j}}{\hbtheta_j}} \}= 0  $  for all $i<j$. Given A.\ref{assume:Stepsize},  \ref{item:as2}  is also satisfied. The a.s. convergence from $\hbtheta_k$ to $\bvartheta$ is arrived. 
\end{proof}

\begin{proof}[Proof for Theorem~\ref{thm:StrongConvergence-2}]
	Let us first show part \ref{item:ub}.  Under A.\ref{assume:Convexity}, 	we have 
	\begin{align}
	& \E _k\ilbracket{\loss\ilparenthesis{\hbtheta_k}}\nonumber\\
	&\,\, \stackrel{\text{a.s.}}{\le} \E _k  \set{\loss\ilparenthesis{\hbtheta_k} - \gain_k\ilbracket{\bg\ilparenthesis{\hbtheta_k}}^\transpose\hbg_k\ilparenthesis{\hbtheta_k} + \frac{\BoundSecondOrder\gain_k^2}{2}\norm{\hbg_k\ilparenthesis{\hbtheta_k}}^2} \label{eq:conv1}\\
	&\,\, \stackrel{\text{a.s.}}{\le} \loss\ilparenthesis{\hbtheta_k} - \gain_k\norm{\bg\ilparenthesis{\hbtheta_k}}^2 + \gain _k O(\perturb_k^2) \norm{\bg\ilparenthesis{\hbtheta_k}} + \frac{\BoundSecondOrder \gain_k^2}{2}\bracket{ O\ilparenthesis{\perturb_k^2} + O(\perturb_k^{-2}) + O(\norm{\bg\ilparenthesis{\hbtheta_k}})^2 } \label{eq:conv2} \\
	&\,\, \stackrel{\text{a.s.}}{=} \loss\ilparenthesis{\hbtheta_k} - \gain_k\norm{\bg\ilparenthesis{\hbtheta_k}}^2 + O(\gain_k\perturb_k^2)\norm{\bg\ilparenthesis{\hbtheta_k}} + O(\gain_k^2\perturb_k^2))  + O\parenthesis{\frac{\gain_k^2}{\perturb_k^2}} + O\ilparenthesis{\gain_k^2} \norm{\bg\ilparenthesis{\hbtheta_k}}^2\nonumber\\
	&\,\, \stackrel{\text{a.s.}}{\le }
	\loss\ilparenthesis{\hbtheta_k} - \frac{\gain_k}{2}\parenthesis{\norm{\bg\ilparenthesis{\hbtheta_k} } - O(\perturb_k^2) }^2 + O\parenthesis{\gain_k^2\perturb_k^2} + O(\gain_k^2\perturb_k^{-2})\,,     \text{for large $k$ s.t.  $ O(\gain_k)<\nicefrac{1}{2} $}\,,
	\label{eq:conv3}
	\end{align}
	where (\ref{eq:conv1}) uses A.\ref{assume:Convexity} and mean-value theorem, (\ref{eq:conv2}) uses Cauchy-Schwartz inequality and  (\ref{eq:Tradeoff7})--(\ref{eq:Tradeoff5}), and (\ref{eq:conv3}) uses A.\ref{assume:Stepsize}. 
	
	Therefore, for sufficiently large $k$, we have
	\begin{align}&
	\E_k \ilbracket{\loss\ilparenthesis{\hbtheta_k} - \loss\ilparenthesis{\bvartheta}}  \stackrel{\text{a.s.}}{ \le } \loss\ilparenthesis{\hbtheta_k} - \loss\ilparenthesis{\bvartheta} + O(\gain_k^2\perturb_k^2 ) + O(\gain_k^2\perturb_k^{-2})\quad - \frac{\gain_k}{2}\parenthesis{\norm{\bg\ilparenthesis{\hbtheta_k}}- O(\perturb_k^2)}^2\,,
	\label{eq:conv4}	\end{align} 
	Under A.\ref{assume:Convexity} and A.\ref{assume:Stepsize}, \cite[Thm. 1]{lai1989extended} ensures that the nonnegative random variable $ \ilbracket{\loss\ilparenthesis{\hbtheta_k} - \loss\ilparenthesis{\bvartheta}} $ converges to a \emph{finite} random variable on a.s. Now that A.\ref{assume:Convexity} assumes $ \sup\set{\norm{\btheta}: \loss\ilparenthesis{\btheta}\le \loss\ilparenthesis{\bvartheta}+\constNum_1} $, the boundedness of $\loss\ilparenthesis{\hbtheta_k}$ a.s. implies the iterate boundedness $ \sup_{k}\norm{\hbtheta_k}<\infty $ a.s. 
	
	Next we show part \ref{item:as}. When (\ref{eq:conv4}) hold, \cite{robbins1971convergence} ensures that $ \lim_{k\to\infty} \sum_{i=1}^{k} \gain_i \ilbracket{ \norm{\bg\ilparenthesis{\hbtheta_i}} - O\ilparenthesis{\perturb_i^2} }^2<\infty $ a.s.  Together with A.\ref{assume:Stepsize}, we have $ \norm{\bg\ilparenthesis{\hbtheta_k}}\to 0  $ as $k\to\infty$ a.s. 
	
	For any \emph{fixed} sample point within a subset of $ \Omega\times\Omega_{\bDelta}  $ with a measure of $1$, the sequence $ \ilset{\hbtheta_0,\cdots,\hbtheta_k,\cdots} $  is a bounded sequence per \ref{item:ub}. By Bolzano-Weierstrass theorem, we can pick a sub-sequence $ \ilset{\hbtheta_{k_0},\cdots,\hbtheta_{k_i},\cdots} $ such that $ \norm{\bg\ilparenthesis{\hbtheta_{k_i}}}\to \zero ^+ $ as $i\to\infty$ a.s. Moreover, the fact that $ \norm{\bg\ilparenthesis{\hbtheta_k}}\to 0 $ a.s.  and the smoothness of $ \bg\ilparenthesis{\cdot}$ ensure that  the limit point of the sub-sequence $ \ilset{\hbtheta_{k_0},\cdots,\hbtheta_{k_i},\cdots} $  as $i\to\infty$ coincides with the limit point of the entire sequence $ \ilset{\hbtheta_0,\cdots,\hbtheta_k,\cdots} $  as $k\to\infty$. Finally, A.\ref{assume:Convexity} asserts that $\bvartheta$ is the unique minimizer such that all neighboring points around it have nonzero gradient evaluation, so the claim in \ref{item:as} is shown. 
\end{proof}

\begin{proof}
	[Proof for Theorem~\ref{thm:StrongConvergenceHessian}] 	First consider the term $ \tperturb_k^{-1}\overline{\ell}_k \mapping_k\ilparenthesis{\tbDelta_k}  $.  
	\begin{align}
	&\E \ilparenthesis{\given{ \tperturb_k^{-1}\overline{\ell}_k \mapping_k\ilparenthesis{\tbDelta_k} }{\hbtheta_k,\bDelta_k}}  \stackrel{\text{a.s.}}{=} \bg\ilparenthesis{\hbtheta_k + \perturb_k\bDelta_k} - \bg\ilparenthesis{\hbtheta_k-\perturb_k\bDelta_k} + O(\perturb_k^3) \,,  \label{eq:HARP-bias1}
	\end{align}
	where the $O(\perturb_k^3)$ term in (\ref{eq:HARP-bias1}) is the difference of the two $O(\perturb_k^2)$ bias terms in the one-sided  gradient approximations for $\bg\ilparenthesis{\hbtheta_k \pm  \perturb_k\bDelta_k}$  in $\tperturb_k^{-1}\overline{\ell}_k \mapping_k\ilparenthesis{\tbDelta_k}$  and $ \tperturb_k = O(\perturb_k) $. Hence, by an expansion of each of $\bg\ilparenthesis{\hbtheta_k \pm \perturb_k\bDelta_k}$, we have for any $i$, $j$
	\begin{align}
	& \E  \parenthesis{\given{ \frac{\overline{\ell}_k}{2\perturb_k\tperturb_k} \mapping_k\ilparenthesis{\tbDelta_k} \ilbracket{\mapping_k\ilparenthesis{\bDelta_k}}^\transpose }{\field_k,\bDelta_k}}  \stackrel{\text{a.s.}}{=}\bH\ilparenthesis{\hbtheta_k} + O(\perturb_k^2)\,,   \label{eq:HARP-bias2}
	\end{align}
	where (\ref{eq:HARP-bias2}) uses (\ref{eq:HARP-bias1}) and  $\E_k\ilparenthesis{\mapping_k\ilparenthesis{\bDelta_k} \bDelta_k^\transpose} = \bI$  in A.\ref{assume:Perturbation}. Note that 
	the $O(\perturb_k^2)$ term in (\ref{eq:HARP-bias2})   absorbs higher-order terms in the Taylor  expansion of $ \bg\ilparenthesis{\hbtheta_k + \perturb_k\bDelta_k} - \bg\ilparenthesis{\hbtheta_k-\perturb_k\bDelta_k}  $ in (\ref{eq:HARP-bias1}).

	Another symmetrization  operation of $ \ilparenthesis{2\perturb_k\tperturb_k}^{-1} \overline{\ell}_k\mapping_k\ilparenthesis{\tbDelta_k} \ilbracket{\mapping_k\ilparenthesis{\bDelta_k}}^\transpose $ gives the latter part of (\ref{eq:HessianUpdate}), in order to  ensure a symmetric Hessian estimate. 
	
	Given (\ref{eq:HARP-bias2}), the statement that $\obH_k  \stackrel{\mathrm{a.s.}}{\longrightarrow}\bH\ilparenthesis{\bvartheta}$ follows from the  Theorem~\ref{thm:StrongConvergence} or Theorem~\ref{thm:StrongConvergence-2},  the updating recursion (\ref{eq:HessianUpdate}),  the algorithmic form in Algorithm~\ref{algo:HARP} and the corresponding analysis in \cite{zhu2019efficient}. 
\end{proof}

\begin{proof}[Proof for Theorem~\ref{thm:AsymptoticNormality}]
	The asymptotic normality result will be shown once the conditions (2.2.1), (2.2.2), and (2.2.3) of \cite{fabian1968asymptotic} hold.  
	
	We first show that  \cite[Eq. (2.2.1)]{fabian1968asymptotic} hold.  We see that $\bGamma_k\to \gain \bH\ilparenthesis{\bvartheta}$ a.s. by the result in Thm.~\ref{thm:StrongConvergence}  and the continuity of $\bH\ilparenthesis{\cdot}$ as assumed in A.\ref{assume:Loss}. 
	When $\upalpha<6\upgamma$, we have $ \bm{t}_k\to \zero $ a.s., as Lemma~\ref{lem:Tradeoff} shows that $ \norm{\bias_k\ilparenthesis{\hbtheta_k} } = O(\perturb_k^2) = O(k^{-2\upgamma}) $ a.s.  When $\upalpha=6\upgamma$, using A.\ref{assume:Perturbation} and Thm.~\ref{thm:StrongConvergence}, we know that $ { \bm{t}_k} = -\gain (k+1)^{2\upgamma} \cdot O(\perturb_k^2) = O(1) $.  
	Using (\ref{eq:bias}), A.\ref{assume:Loss}, and  Thm.~\ref{thm:StrongConvergence}, we have 
	\begin{equation}\label{eq:limitT1} \bias_k \stackrel{k\to\infty}{\longrightarrow} \frac{1}{6} \perturb_k^2  \E \ilbracket{\loss^{(3)} \ilparenthesis{\bvartheta }\cdot  \ilparenthesis{\bDelta\otimes\bDelta\otimes\bDelta}\cdot \mapping \ilparenthesis{\bDelta} } \, \text{a.s.} \,,
	\end{equation}
	thanks to the  dominated convergence theorem. 
	Multiplying $ -\gain (k+1)^{\nicefrac{\uptau}{2}} = -\gain (k+1)^{2\upgamma} $ on both sides of  (\ref{eq:limitT1}) gives (\ref{eq:bias1}).  
	Combined the cases for $\upalpha<6\upgamma$ and $\upalpha=6\upgamma$,  we know that $ \bm{t}_k  $ converges to a finite vector   for $\upalpha\le6\upgamma$. 
	
	We then show that \cite[Eq. (2.2.2)]{fabian1968asymptotic} hold.
	By definition (\ref{eq:gDecomposition}),  $\noise_k\ilparenthesis{\hbtheta_k}$ is a $\field_k$-measurable martingale  sequence, and so is $ \bv_k  $.  
	\begin{align}
	& \E _k\ilparenthesis{\bv_k\bv_k^\transpose} \nonumber\\
	&\,\,  \stackrel{\text{a.s.}}{=}
	\frac{\gain^2}{\ilparenthesis{k+1}^{2\upgamma}}
	\big(  \E_k \ilset{ \hbg_k\ilparenthesis{\hbtheta_k} \ilbracket{\hbg_k\ilparenthesis{\hbtheta_k}}^\transpose }     - \E_k \ilbracket{\hbg_k\ilparenthesis{\hbtheta_k}}  \ilset{   \E_k \ilbracket{\hbg_k\ilparenthesis{\hbtheta_k}}   }^\transpose \big)  \label{eq:limitVV1}\\
	&\,\,  \stackrel{\text{a.s.}}{=}   \frac{\gain^2}{\perturb^2}\perturb_k^2\E_k \ilset{  \hbg_k\ilparenthesis{\hbtheta_k } \ilbracket{ \hbg_k\ilparenthesis{\hbtheta_k } }^\transpose}   +\frac{\gain^2}{\perturb^2}\perturb_k^2 \ilbracket{ \bg_k\ilparenthesis{\hbtheta_k} + \bias_k\ilparenthesis{\hbtheta_k} }\ilbracket{ \bg_k\ilparenthesis{\hbtheta_k} + \bias_k\ilparenthesis{\hbtheta_k} }^\transpose \label{eq:limitVV2} \\ 
	&\,\, \stackrel{\text{a.s.}}{=} \frac{\gain^2}{\perturb^2}  \cdot \E _k \bracket{ \parenthesis{\frac{\upvarepsilon_k^+-\upvarepsilon_k^-}{2 }}^2 \mapping_k\ilparenthesis{\bDelta_k} \ilbracket{\mapping_k\ilparenthesis{\bDelta_k}}^\transpose  }    + o(1)\nonumber\\
	& \,\, \stackrel{\text{a.s.}}{=} \frac{\gain^2}{4\perturb^2}\E _k \set{ \mapping_k\ilparenthesis{\bDelta_k} \ilbracket{\mapping_k\ilparenthesis{\bDelta_k}}^\transpose\E\ilbracket{\given{\ilparenthesis{\upvarepsilon_k^+-\upvarepsilon_k^-}^2}{\hbtheta_k,\bDelta_k}}  }    + o(1)\nonumber\\
	&\,\, \stackrel{\text{a.s.}}{=} \frac{\gain^2}{\perturb^2} \frac{2\Var\ilbracket{\ell\ilparenthesis{\bvartheta,\upomega}}}{4}  \E \ilset{\mapping_k\ilparenthesis{\bDelta_k}\ilbracket{\mapping_k\ilparenthesis{\bDelta_k}}^\transpose}+ o(1)\label{eq:limitVV3}\\
	&\,\,\stackrel{\text{a.s.}}{\longrightarrow}  \frac{\gain^2 \Var\ilbracket{\ell\ilparenthesis{\bvartheta,\upomega}}}{2\perturb^2} \bSigma\,, \text{ as }k\to\infty\,, \label{eq:limitVV4}
	\end{align}
	where (\ref{eq:limitVV1}) follows from (\ref{eq:gDecomposition}), the $o(1)$ term on (\ref{eq:limitVV2}) is due to  A.\ref{assume:Perturbation},   (\ref{eq:bias}), Lemma~\ref{lem:Tradeoff}, and Theorem~\ref{thm:StrongConvergence}, 
	both (\ref{eq:limitVV3}) and (\ref{eq:limitVV4}) are due to A.\ref{assume:AdditionalNormality} and Remark~\ref{rem:NoiseCMC}.

	We finally show that either (2.2.3) or (2.2.4) in  \cite{fabian1968asymptotic} hold. That is, for every $\upeta>0$, $ \lim_{k\to\infty} \E  ({ \norm{\bv_k}^2 \indicator _{\set{\norm{\bv_k}^2\ge \upeta k^{\upalpha}}} })   = 0 $.  For any $\constNum_5\in \ilparenthesis{0,\constNum_4/2}$, we have 
	\begin{align}
	&\lim_{k\to\infty} \E \parenthesis{ \norm{\bv_k}^2 \indicator _{\set{\norm{\bv_k}^2\ge \upeta k^{\upalpha}}} } \nonumber\\
	&\,\, \le \limsup_{k\to\infty} \ilbracket{\Prob\ilparenthesis{ \norm{\bv_k}^2\ge \upeta k^{\upalpha} }}^{\frac{\constNum_5}{1+\constNum_4}} \cdot \ilbracket{\E\ilparenthesis{\norm{\bv_k}^{2\ilparenthesis{1+\constNum_5}}}}^{\frac{1}{1+\constNum_5}} \nonumber\\
	& \,\, \le   \limsup_{k\to\infty}  \parenthesis{ \frac{\E \ilparenthesis{\norm{\bv_k}^2}}{\upeta k^{\upalpha}} }^{\frac{\constNum_5}{1+\constNum_4}} \cdot \ilbracket{ \E\ilparenthesis{\norm{\bv_k}^{2\ilparenthesis{1+\constNum_5}}}} ^{\frac{1}{1+\constNum_5}}\,,
	\end{align}
	where the first inequality is due to Holder's inequality and the second inequality is due to Markov's inequality.
	
	Using Minkowski inequality, we have   $ \norm{\bv_k}^{2\ilparenthesis{1+\constNum_5}}\le 2(1+\constNum_5) k^{-2\ilparenthesis{1+\constNum_5}\upgamma} \big[ \norm{ \hbg_k\ilparenthesis{\hbtheta_k}^{2\ilparenthesis{1+\constNum_5}} + \norm{\bg\ilparenthesis{\hbtheta_k}}^{2\ilparenthesis{1+\constNum_5}} + \norm{\bias_k\ilparenthesis{\hbtheta_k}}^{2\ilparenthesis{1+\constNum_5}}    }  \big]$. From Lemma~\ref{lem:Tradeoff} and A.\ref{assume:ODE}, we know that there exists some  $K$ such that both $\bias_k\ilparenthesis{\hbtheta_k}$ and $\bg\ilparenthesis{\hbtheta_k}$ are uniformly    bounded a.s. for all $k\ge K$. Lemma~\ref{lem:Tradeoff} also implies that $ \norm{\hbg_k\ilparenthesis{\hbtheta_k}}= O(\perturb_k^{-2}) $. Combined, we have $ \E\norm{\bv_k}^{2\ilparenthesis{1+\constNum_5}} = O(1) $. 
	
	Now that all relevant conditions in \cite{fabian1968asymptotic} are met to ensure the asymptotic normality. 
\end{proof}

\begin{proof}[Proof of Lemma~\ref{lem:CRN}]
	
	Under A.\ref{assume:NoiseCRN}, 
	\begin{align}
	& \E _k \ilbracket{\hbg_k\ilparenthesis{\hbtheta_k} \ilbracket{\hbg_k\ilparenthesis{\hbtheta_k} }^\transpose}\nonumber\\
	&\,\,\stackrel{\text{a.s.}}{=}  \frac{1}{4\perturb_k^2 }   \E_k  \bigg\{  \mapping_k\ilparenthesis{\bDelta_k} \ilbracket{\mapping_k\ilparenthesis{\bDelta_k}}^\transpose \times  \ilbracket{\ell\ilparenthesis{\hbtheta_k + \perturb_k\bDelta_k,\upomega_k} - \ell\ilparenthesis{\hbtheta_k -\perturb_k\bDelta_k,\upomega_k}}^2 \bigg\}\nonumber\\
	&\,\,\stackrel{\text{a.s.}}{=}  \frac{1}{4\perturb_k^2 }   \E_k  \bigg\{  \mapping_k\ilparenthesis{\bDelta_k} \ilbracket{\mapping_k\ilparenthesis{\bDelta_k}}^\transpose  \times  \E\ilbracket{\given{\ilbracket{\ell\ilparenthesis{\hbtheta_k + \perturb_k\bDelta_k,\upomega_k} - \ell\ilparenthesis{\hbtheta_k -\perturb_k\bDelta_k,\upomega_k}}^2}{\hbtheta_k, \bDelta_k}} \bigg\}\,. \label{eq:CRN1}
	\end{align}
	Similar to the third-order Taylor expansion in Lemma~\ref{lem:Tradeoff}, we have 
	\begin{align}
	&\frac{1}{4\perturb_k^2 }\E\ilbracket{\given{\ilbracket{\ell\ilparenthesis{\hbtheta_k + \perturb_k\bDelta_k,\upomega_k} - \ell\ilparenthesis{\hbtheta_k -\perturb_k\bDelta_k,\upomega_k}}^2}{\hbtheta_k, \bDelta_k}}\nonumber\\
	&\,\, \stackrel{\text{a.s.}}{=} \E \set{\given{\ilbracket{\bDelta_k ^\transpose \noisyG\ilparenthesis{\hbtheta_k,\upomega_k}}^2}{\hbtheta_k,\bDelta_k}} + O(\perturb_k^4)\nonumber\\
	&\,\, \stackrel{\text{a.s.}}{=}  \bracket{\bDelta_k ^\transpose \noisyG\ilparenthesis{\hbtheta_k,\upomega_k}}^2 + O(\perturb_k^4)\,.
	\end{align} Whence, (\ref{eq:CRN1}) becomes
	\begin{align}
	&\E _k \ilbracket{\hbg_k\ilparenthesis{\hbtheta_k} \ilbracket{\hbg_k\ilparenthesis{\hbtheta_k} }^\transpose} \nonumber\\
	&\,\, \stackrel{\text{a.s.}}{=} \E _k \set{ \mapping_k\ilparenthesis{\bDelta_k}   \bDelta_k ^\transpose \noisyG\ilparenthesis{\hbtheta_k,\upomega_k}  \ilbracket{\noisyG\ilparenthesis{\hbtheta_k,\upomega_k}}^\transpose \bDelta_k  \ilbracket{\mapping_k\ilparenthesis{\bDelta_k}}^\transpose }  + o(1)\,. \label{eq:CRN2}
	\end{align}
	Now that A.\ref{assume:Perturbation} assumes independence between $\hbtheta_k$ and $\bDelta_k$, then the $ (i,j)-$th component of (\ref{eq:CRN2}) equals the following a.s.:    
	\begin{align}
	& \E \bracket{\sum_{p=1}^{\Dimension} \sum_{q=1}^{\Dimension} m_{k,i} \Delta_{k,p} \Delta_{k,q} m_{k,j} } \cdot \E_k \parenthesis{ \noisyGcomponent_{k,p}\noisyGcomponent_{k,q}} + o(1)\nonumber\\
	& \,\, \stackrel{\text{a.s.}}{=} \bracket{\indicator_{\ilset{i=j}} \indicator_{\ilset{p = q}}   + \indicator_{\ilset{i\neq j }}\ilparenthesis{ \indicator_{\ilset{ p=i, q = j }} + \indicator_{ \ilset{ p = j, q = i  }}} }  \times   \E_k \parenthesis{ \noisyGcomponent_{k,p}\noisyGcomponent_{k,q}   } + o(1)
	\label{eq:CRN3}\\
	&\,\,\stackrel{\text{a.s.}}{=}\begin{cases}
	\sum_{p=1}^{\Dimension}  \E _k \ilparenthesis{  \noisyGcomponent_{k,p}   }^2 + o(1),  &\text{if } i = j,\\ 
	2\E_k \ilparenthesis{ \noisyGcomponent_{k, i}   \noisyGcomponent_{k, j}  }  + o(1), & \text{if }i\neq j.
	\end{cases}\label{eq:CRN4}
	\end{align}
	where $ m_{k,i} $ is the $i$th component of $ \mapping_k\ilparenthesis{\bDelta_k} $,   $ \Delta_{k,p} $ is the $p$th component of $\bDelta_k$, $\noisyGcomponent_{k,p}$ is the $p$th component of $ \noisyG\ilparenthesis{\hbtheta_k,\upomega_k} $,   equality (\ref{eq:CRN3}) uses $ \E_k \ilbracket{\mapping_k\ilparenthesis{\bDelta_k}\bDelta_k^\transpose} = \bI $ in A.\ref{assume:Perturbation}. Taking the diagonal terms  of (\ref{eq:CRN4})    
	gives    (\ref{eq:NoiseCRN-1}). 
\end{proof}

\begin{proof}[Proof for Theorem~\ref{thm:AsymptoticNormality-CRN}]
	
	We first show that \cite[Eq. (2.2.1)]{fabian1968asymptotic} hold. As in the proof for Thm.~\ref{thm:AsymptoticNormality}, $\bGamma_k\to \gain\bH\ilparenthesis{\bvartheta}$ a.s. When $ \upalpha<  4\upgamma $, $ \bm{t}_k = O(k^{\nicefrac{\upalpha}{2}-2\upgamma})  \to \zero $. Hence,  \cite[Eq. (2.2.1)]{fabian1968asymptotic}  is met. 
	
	We then show \cite[Eq. (2.2.2)]{fabian1968asymptotic} hold.  Following the same reasoning as (\ref{eq:limitVV2}), we have 
	\begin{align}  \E_k\ilparenthesis{\bv_k\bv_k^\transpose} \stackrel{\text{a.s.}}{=} \gain ^2  \E_k\ilset{\hbg_k\ilparenthesis{\hbtheta_k} \ilbracket{\hbg_k\ilparenthesis{\hbtheta_k}} ^\transpose}   + o(1)\,, 
	\end{align} which is exactly (\ref{eq:CRN4}). 
	Under A.\ref{assume:NoiseCRN},  $\upomega_k$ is independent from $\field_k$, we have  $\E_k \noisyGcomponent_{k,p}^2 \stackrel{\text{a.s.}}{=} \E \noisyGcomponent_{k,p}^2 \stackrel{\text{a.s.}}{=}  \int_{\upomega\in \Omega } \ilbracket{\noisyGcomponent\ilparenthesis{\hbtheta_k,\upomega}}_{p}^2 \diff \Prob\ilparenthesis{\upomega}  \stackrel{\text{a.s.}}{\longrightarrow}   \int_{\upomega\in \Omega } \ilbracket{\noisyGcomponent\ilparenthesis{\bvartheta,\upomega}}_{p}^2 \diff \Prob\ilparenthesis{\upomega}   $
	as $k\to\infty$, where the asymptotic relationship is due to dominated convergence theorem and A.\ref{assume:NoiseCRN}. Following the same line of reasoning, $ 
	\E_k \ilparenthesis{ \noisyGcomponent_{k, i}   \noisyGcomponent_{k, j}  }    \stackrel{\text{a.s.}}{=}  \int_{\upomega\in \Omega } \ilbracket{\noisyGcomponent\ilparenthesis{\hbtheta_k,\upomega}}_{i}  \ilbracket{\noisyGcomponent\ilparenthesis{\hbtheta_k,\upomega}}_{j}  \diff \Prob\ilparenthesis{\upomega}  \stackrel{\text{a.s.}}{\longrightarrow}   \int_{\upomega\in \Omega } \ilbracket{\noisyGcomponent\ilparenthesis{\bvartheta,\upomega}}_{i}  \ilbracket{\noisyGcomponent\ilparenthesis{\bvartheta,\upomega}}_{j}  \diff \Prob\ilparenthesis{\upomega}   $. Combined,  we have (\ref{eq:noisyGcov}). 
	
	The proof of showing \cite[Eq. (2.2.3)]{fabian1968asymptotic} is exactly the same as that in proof for Theorem~\ref{thm:AsymptoticNormality}.  
\end{proof}

\end{document}